\pdfoutput=1
\RequirePackage{ifpdf}
\ifpdf 
\documentclass[pdftex]{sigma}
\else
\documentclass{sigma}
\fi

\numberwithin{equation}{section}

\newtheorem{Theorem}{Theorem}[section]
\newtheorem*{Theorem*}{Theorem}
\newtheorem{Corollary}[Theorem]{Corollary}
\newtheorem{Lemma}[Theorem]{Lemma}
\newtheorem{Proposition}[Theorem]{Proposition}
 { \theoremstyle{definition}
\newtheorem{Definition}[Theorem]{Definition}
\newtheorem{Example}[Theorem]{Example}
\newtheorem{Remark}[Theorem]{Remark} }

\usepackage{tabularx}
\usepackage{tikz,tikz-cd}
\usetikzlibrary{snakes, matrix,shapes,arrows,positioning,chains}

\newcommand{\g}{\mathfrak{g}}
\newcommand{\Id}{\operatorname{Id}}

\newcommand{\A}{\mathcal{A}}

\newcommand{\NC}{\operatorname{NC}}
\newcommand{\Int}{\operatorname{Int}}

\newcommand{\ST}{\operatorname{ST}}

\newcommand{\cHS}{H_\mathcal{S}}
\newcommand{\cR}{\mathcal{R}}
\newcommand{\ccP}{\mathcal{P}}
\newcommand{\uno}{{\bf{1}}}
\newcommand{\sk}{\operatorname{sk}}

\newcommand{\AdmT}{\widetilde{\operatorname{Adm}}}

\newcommand{\Expl}{\mathcal{E}_\prec}
\newcommand{\Expr}{\mathcal{E}_\succ}

\usepackage{forest}
\forestset{
 decor/.style = {
 label/.expanded = {[inner sep = 0.2ex, font=\unexpanded{\tiny}]right:{$#1$}}
 },
 root/.style = {minimum size = 0.1ex},
 decorated/.style = {
 for tree = {
 circle, fill, inner sep = 0.3ex, minimum size = 1.ex,
 grow' = south, l = 0, l sep = 1.2ex, s sep = 0.7em,
 fit = tight, parent anchor = center, child anchor = center,
 delay = {decor/.option = content, content =}
 }
 },
 default preamble = {decorated, root},
 begin draw/.code={\begin{tikzpicture}[baseline={([yshift=-0.5ex]current bounding box.center)}]},
}

\definecolor{red}{rgb}{1.,0.,0.}
\definecolor{green}{rgb}{0.,1.,0.}
\definecolor{blue}{rgb}{0.,0.,1.}
\definecolor{orange}{rgb}{1.,0.8431372549019608,0.}

\tikzset{
solid node/.style={circle,draw,inner sep=1.5,fill=black},
hollow node/.style={circle,draw,inner sep=1.5,fill=white}
}

\begin{document}
\allowdisplaybreaks

\newcommand{\arXivNumber}{2111.02179}

\renewcommand{\thefootnote}{}

\renewcommand{\PaperNumber}{073}

\FirstPageHeading

\ShortArticleName{Monotone Cumulant-Moment Formula and Schr\"oder Trees}

\ArticleName{Monotone Cumulant-Moment Formula\\ and Schr\"oder Trees\footnote{This paper is a~contribution to the Special Issue on Non-Commutative Algebra, Probability and Analysis in Action. The~full collection is available at \href{https://www.emis.de/journals/SIGMA/non-commutative-probability.html}{https://www.emis.de/journals/SIGMA/non-commutative-probability.html}}}

\Author{Octavio ARIZMENDI~$^{\rm a}$ and Adrian CELESTINO~$^{\rm b}$}

\AuthorNameForHeading{O.~Arizmendi and A.~Celestino}

\Address{$^{\rm a)}$~Centro de Investigaci\'on en Matem\'aticas, Calle Jalisco SN. Guanajuato, Mexico}
\EmailD{\href{octavius@cimat.mx}{octavius@cimat.mx}}
\URLaddressD{\url{https://www.cimat.mx/~octavius/}}

\Address{$^{\rm b)}$~Department of Mathematical Sciences, Norwegian University of Science and Technology,\\
\hphantom{$^{\rm b)}$}~7491 Trondheim, Norway}
\EmailD{\href{mailto:adrian.celestino@ntnu.no}{adrian.celestino@ntnu.no}}

\ArticleDates{Received March 23, 2022, in final form September 25, 2022; Published online October 07, 2022}

\Abstract{We prove a formula to express multivariate monotone cumulants of random variables in terms of their moments by using a Hopf algebra of decorated Schr\"oder trees.}


\Classification{05E99; 16T05; 17A30; 46L53}

\renewcommand{\thefootnote}{\arabic{footnote}}
\setcounter{footnote}{0}

\section{Introduction}

In the theory of noncommutative probability, there exist different notions of independence that replace in a noncommutative context the classical notion of independence of random variables in a probability space $(\Omega,\mathcal{F},\mathbb{P})$. By defining in a precise way what a good notion of independence should be, we find that there are only five such notions \cite{BGS,Mura03,Spe97}: tensorial, free, Boolean, monotone and anti-monotone. Free independence was introduced by Voiculescu \cite{Voi85}. It is the most studied one and has fruitful relations and applications with other branches of mathematics, such as combinatorics and random matrix theory.

For many notions in classical probability theory, it is possible to find the respective noncommutative counterpart for each of the distinct noncommutative notions of independence. For example, we have convolutions as well as central limit theorems for the aforementioned five notions of independence. One of the most important concepts is that of cumulants, which have proven to provide the major tool in the combinatorial study of noncommutative probability.

Cumulants are combinatorially defined implicitly in terms of the so-called \textit{moment-cumulant relations}
\begin{alignat*}{3}
& m_n = \sum_{\pi\in \NC(n)} r_\pi \qquad && \text{(free case)},& \\
& m_n = \sum_{\pi\in \Int(n)} b_\pi \qquad && \text{(Boolean case)}, &\\
& m_n = \sum_{\pi\in \NC(n)} \frac{1}{t(\pi)! }h_\pi \qquad && \text{(monotone case)}.&
\end{alignat*}
Here the sums are over certain lattices of noncrossing partitions, which allow expressing the moments as polynomials of the cumulants; see Section~\ref{sec:momentcumulantrel} for the precise definitions. For the free and Boolean case, it is possible to invert the corresponding formulas via general M\"obius inversion on the corresponding partially ordered sets in order to obtain formulas that express free and Boolean cumulants in terms of the moments
\begin{gather}
	\label{eq:freeCM}
	r_n(a_1,\dots,a_n)
	= \sum_{\sigma \in \NC(n)} \operatorname{M\ddot{o}b}(\sigma,1_n)\varphi_\sigma(a_1,\dots,a_n),
	\\ \label{eq:BoolCM}
	b_n(a_1,\dots,a_n)
	= \sum_{\sigma \in \Int(n)}(-1)^{|\sigma|-1}\varphi_\sigma(a_1,\dots,a_n),
\end{gather}
where $\mathrm{M\ddot{o}b}$ stands for the M\"obius function on the incidence algebra of noncrossing partitions. However, this cannot be performed for the monotone analogue since the coefficients $\frac{1}{t(\pi)!}$ do not satisfy certain multiplicative conditions required for M\"obius inversion.

There has been some interest in finding an explicit formula for expressing monotone cumulants in terms of moments using noncrossing partitions. To be specific, one is interested in finding coefficients $\alpha(\pi)$ for any noncrossing partition $\pi$ such that a formula of the form
\begin{equation}\label{eq:objective}
h_n = \sum_{\pi\in \NC(n)} \alpha(\pi)m_\pi,
\end{equation}
holds for each $n\geq 1$.

In this paper, we propose instead to use the viewpoint of Josuat-Verg\`es, Menous, Novelli, and Thibon put forward in \cite{JMNT}, where these authors introduce a Hopf algebra $\cHS(\A)$ of decorated Schr\"oder trees to rewrite formula \eqref{eq:freeCM} as a sum over of Schr\"oder trees. Our main result gives an analogous formula for expressing monotone cumulants in terms of the moments as a sum indexed by Schr\"oder trees. See Sections \ref{sec:schroder} and \ref{sec:monotone} for the precise definitions appearing in the formulas.

\begin{Theorem}\label{thm:mainthm1}
Let $(\A,\varphi)$ be a noncommutative probability space and let $\{h_n\colon \A^n\to\mathbb{C}\}_{n\geq1}$ be the family of monotone cumulants of $(\A,\varphi)$. Then, for any $n\geq1$ and elements $a_1,\dots,a_n\in \A$, we have that
\begin{equation}\label{eq:MonMom1}
	h_n(a_1,\dots,a_n) =\sum_{t\in \ST(n)} \omega(\sk(t))\varphi_{\pi(t)} (a_1,\dots,a_n),
\end{equation}
where $\ST(n)$ denotes the set of Schr\"oder trees with $n+1$ leaves and, for a Schr\"oder tree $t$, $\sk(t)$ denotes the so-called skeleton of $t$, and $\omega(\sk(t))$ is the Murua coefficient associated to~$\sk(t)$.
\end{Theorem}

We remark that \eqref{eq:MonMom1} may be rewritten in the form of \eqref{eq:objective} by regrouping the trees according to which $\pi\in \NC(n)$ satisfies that $\pi(t)=\pi$. Thus \eqref{eq:MonMom1} may be seen as a refinement of the desired formula~\eqref{eq:objective} (see Corollary~\ref{cor:maincor} for the precise statement).

Our approach is based on a series of recent papers by Ebrahimi-Fard and Patras \cite{EFP1,EFP2,EFP3,EFP4,EFP5}, where the authors have developed a Lie-theoretical framework for cumulants in noncommutative probability. In such a picture, the relations between moments and the different types of cumulants can be explained analogously to the relation between a Lie algebra and its corresponding (Lie) group. More precisely, given a noncommutative probability space $(\A,\varphi)$, we can consider a~Hopf algebra $H$ given by the double tensor algebra defined over $\A$, and a character $\Phi\colon H \to \mathbb{C}$ extending the linear functional $\varphi$. The free, Boolean, and monotone cumulant functionals can be identified with certain infinitesimal characters $\kappa$, $\beta$ and $\rho$ defined on $H$, respectively. The important feature of~$H$ is that its coproduct splits into two half-coproducts, providing two non-associative products $\prec$ and $\succ$ on the algebraic dual $H^*$, such that $(H^*,\prec,\succ)$ is a unital noncommutative shuffle algebra~\cite{EFP1}.

Within this framework, the relations between moments and free, Boolean, and monotone cumulants are encoded by the three exponential maps $\mathcal{E}_\prec$, $\mathcal{E}_\succ$ and $\exp^*$ associated to $\prec$, $\succ$ and the associative product $*$ of $H^*$, respectively:
\begin{gather*}
\Phi = \mathcal{E}_\prec(\kappa),
\qquad \Phi = \mathcal{E}_{\succ}(\beta),
\qquad \Phi = \exp^*(\rho).
\end{gather*}
We observe that in this context, the problem of expressing the monotone cumulants in terms of the moments is simply equivalent to computing the logarithm with respect to the associative product $\rho = \log^*(\Phi).$ However, finding a closed formula in terms of noncrossing partitions is not obvious from this relation.

\subsection*{Organization of the paper} In addition to the preceding introduction, this paper is organized in the following way. Section~\ref{sec:momentcumulantrel} contains the notations and definitions of the basic combinatorial objects of noncommutative probability, such as noncrossing partitions, monotone partitions, and the forest of nesting of a~noncrossing partition. With this, we can give a precise definition of cumulants for free, Boolean and monotone independence. In Section~\ref{sec:shuffle}, we explain the main ideas for the Hopf algebraic approach of Ebrahimi-Fard and Patras. In particular, we give an explicit definition of the shuffle algebra of interest, as well as describe the link with the moment-cumulant formulas. In Section~\ref{sec:schroder}, we present the Hopf algebra of decorated Schr\"oder trees together with the shuffle algebra associated to it. Section~\ref{sec:monotone} is devoted to the main result of this paper: we prove two combinatorial lemmas, and as a consequence, we get Theorem~\ref{thm:mainthm1}. Finally, in Section~\ref{sec:boolean}, we obtain the Boolean counterpart of the moment-cumulant formula in terms of Schr\"oder trees.

\section{Moment-cumulant relations in noncommutative probability}\label{sec:momentcumulantrel}

\subsection{Noncrossing partitions}

Partitions and noncrossing partitions are fundamental objects in the combinatorial study of noncommutative probability theory.
\begin{Definition}\quad
\begin{enumerate}\itemsep=0pt
\item Recall that a \textit{noncrossing partition} of $[n] = \{1,\dots,n\}$ is a set partition $\pi = \{V_1,\dots,V_k\}$ of $[n]$ such that there are no elements $a<b<c<d$ in $[n]$ such that $a,c \in V_i$ and $b,d \in V_j$, and $i\neq j$. The elements of $\pi$ are called the \textit{blocks of $\pi$}. We denote by $\NC(n)$ to the set of noncrossing partitions of $[n]$.
\item We say that $\pi \in \NC(n)$ is \textit{irreducible} if both $1$ and $n$ belong to the same block in $\pi$. We will write $\NC_{\mathrm{irr}}(n)$ to denote the set of irreducible noncrossing partitions on~$[n]$.
\item An \textit{interval partition} is a noncrossing partition $\pi\in\NC(n)$ such that all its blocks are of the form $\{k,k+1,\dots,k+l\}$ for some $1\leq k\leq n$ and $0\leq l\leq n-k$. The set of interval partitions of $[n]$ is denoted $\Int(n)$.
\end{enumerate}
\end{Definition}

\begin{figure}[t]
$$
\begin{array}{c c c}
\begin{tikzpicture}[thick,font=\small]
 \path 	(0,0) 		node (a) {1}
 	(0.5,0) 	node (b) {2}
 	(1,0) 		node (c) {3}
 	(1.5,0) 	node (d) {4}
 	(2,0) 		node (e) {5}
	(2.5,0) 		node (f) {6};
 \draw (a) -- +(0,0.75) -| (e);
 \draw (b) -- +(0,0.60) -| (d);
 \draw (c) -- +(0,0.75) -| (c);
 \draw (f) -- +(0,0.75) -| (f);
 \end{tikzpicture}
 & \quad &
 \begin{tikzpicture}[thick,font=\small]
 \path 	(0,0) 		node (a) {1}
 	(0.5,0) 	node (b) {2}
 	(1,0) 		node (c) {3}
 	(1.5,0) 	node (d) {4}
 	(2,0) 		node (e) {5}
 	(2.5,0) 	node (f) {6};
 \draw (a) -- +(0,0.75) -| (d);
 \draw (c) -- +(0,0.75) -| (c);
 \draw (b) -- +(0,0.60) -| (b);
 \draw (e) -- +(0,0.75) -| (f);
 \end{tikzpicture}\\
 \mbox{a) Crossing partition}& \quad & \mbox{b) Noncrossing partition}\\[0.2cm]
 \begin{tikzpicture}[thick,font=\small]
 \path 	(0,0) 		node (a) {1}
 	(0.5,0) 	node (b) {2}
 	(1,0) 		node (c) {3}
 	(1.5,0) 	node (d) {4}
 	(2,0) 		node (e) {5}
 	(2.5,0) 	node (f) {6};
 \draw (a) -- +(0,0.75) -| (f);
 \draw (b) -- +(0,0.60) -| (e);
 \draw (c) -- +(0,0.45) -| (d);
 \end{tikzpicture}
 & \quad &
 \begin{tikzpicture}[thick,font=\small]
 \path 	(0,0) 		node (a) {1}
 	(0.5,0) 	node (b) {2}
 	(1,0) 		node (c) {3}
 	(1.5,0) 	node (d) {4}
 	(2,0) 		node (e) {5}
 	(2.5,0) 	node (f) {6};
 \draw (a) -- +(0,0.75) -| (b);
 \draw (c) -- +(0,0.75) -| (e);
 \draw (d) -- +(0,0.75) -| (d);
 \draw (f) -- +(0,0.75) -| (f);
 \end{tikzpicture}
 \\ \mbox{c) Irreducible partition} &\quad &\mbox{d) Interval partition}
\end{array}
$$
\caption{Different types of set partitions of the set $[6]=\{1,\dots,6\}$.}
\end{figure}
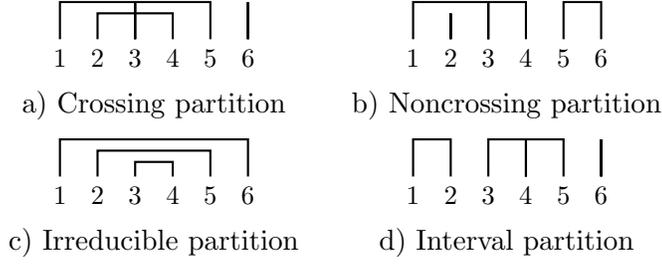

The set $\NC(n)$ is endowed with a poset structure as follows: for two partitions $\pi,\sigma \in \NC(n)$, we say that $\pi \leq \sigma$ if and only if every block in $\pi$ is contained in a block of $\sigma$. This order is called the \textit{reverse refinement order} on $\NC(n)$. It provides a lattice structure on $\NC(n)$, with maximal element being the partition into a single block, $1_n = \{\{1,\dots,n\}\}$, and the minimal element consisting of $n$ blocks, $0_n = \{ \{1\},\dots, \{n\}\}$.

For every noncrossing partition $\pi = \{V_1,\dots,V_k\}\in \NC(n)$, we can define an order on the blocks by stating that for $V_{i_s},V_{i_t} \in \pi$, $V_{i_s}\leq V_{i_t}$ if only if $V_{i_t}$ is nested in $V_{i_s}$, i.e., $a \in [\min{V_{i_s}},\max{V_{i_s}}]$ for any $a \in V_{i_t}$. It is clear that the maximal elements in~$\pi$ are the interval blocks.
Given a minimal block $V\in \pi\in \NC(n)$, we say that $\pi':=\{ W\in \pi\colon W\leq V\}$ is an \textit{irreducible component of~$\pi$}. In other words, $\pi'$ is a noncrossing partition obtained by restricting~$\pi$ to the interval $[\min(V),\max(V)]$. It is clear then that $\pi$ has a unique irreducible component if and only if $\pi$ is an irreducible noncrossing partition.

Associated to an irreducible noncrossing partition $\pi\in \NC_{\mathrm{irr}}(n)$, we can construct a rooted tree $t(\pi)$ with $|\pi| = k$ vertices. The vertices of $t(\pi)$ are decorated by the blocks of $\pi$, and there is a directed edge from a vertex $v$ to a vertex $w$ if and only if the corresponding block $V$ associated to $v$ \textit{covers} the block $W$ associated to $w$, i.e.~$V<W$ and there is no other block $U\in \pi$ such that $V<U<W$. The tree $t(\pi)$ is called the \textit{tree of nestings} of~$\pi$. For the general case of $\pi \in \NC(n)$, we define the \textit{forest of nestings of~$\pi$}, also denoted by~$t(\pi)$. It consists of the ordered forest of rooted trees of nestings of the irreducible components of~$\pi$, where the order of the irreducible components is given by the total order determined by the minimum element of each component.

The \emph{tree factorial} is recursively defined for any rooted tree $t$ by $t! = 1$ if $t$ is the rooted tree consisting of a single vertex. Otherwise, if $t$ is a rooted tree that can be obtained by grafting the subtrees $s_1,\dots,s_m$ to the root vertex, i.e., $t=[s_1,\dots,s_m]$, then
$t! = |s| s_1!\cdots s_m!$.
Here, the degree $|t|=1+|s_1| + \cdots + |s_m|$ of $t$ is its number of vertices. Finally, if $f$ is a forest formed by rooted trees $t_1,\dots,t_m$, we define $f! := t_1!\cdots t_m!$.

\begin{figure}[h]
$$
\begin{array}{c@{\,}c@{\,} c c @{\,}c@{\,} c}
\pi &=& \begin{tikzpicture}[baseline={([yshift=0.3ex]current bounding box.center)},thick,font=\small]
 \path 	(0,0) 		node (a) {1}
 	(0.5,0) 	node (b) {2}
 	(1,0) 		node (c) {\phantom{1}}
 	(1.5,0) 	node (d) {4}
 	(2,0) 		node (e) {\phantom{1}}
 	(2.5,0) 	node (f) {6}
 	(3,0) 		node (g) {\phantom{1}}
 	(3.5,0) 	node (h) {\phantom{1}}
		(4,0) 		node (i) {\phantom{1}}
		(4.5,0) 	node (j) {\phantom{1}};
 \draw (a) -- +(0,0.75) -| (j);
 \draw (b) -- +(0,0.60) -| (b);
 \draw (c) -- +(0,0.75) -| (c);
 \draw (d) -- +(0,0.60) -| (e);
 \draw (e) -- +(0,0.60) -| (h);
 \draw (f) -- +(0,0.50) -| (g);
 \draw (i) -- +(0,0.75) -| (j);
 \end{tikzpicture}, 	& \quad t(\pi) &=&
				 \begin{tikzpicture}[baseline={([yshift=-1ex]current bounding box.center)},scale=0.6,
level 1/.style={level distance=7mm,sibling distance=7mm}]
\node(0)[solid node,label=right:{\tiny 1}]{}
child{node(1)[solid node,label=left:{\tiny 2}]{}
}
child{node(2)[solid node,label=right:{\tiny 4}]{}
	child{[black] node(11)[solid node,label=right:{\tiny 6}]{}}
};
\end{tikzpicture}
\\[5mm]
\sigma &= & \begin{tikzpicture}[baseline={([yshift=0.3ex]current bounding box.center)},thick,font=\small]
 \path 	(0,0) 		node (a) {1}
 	(0.5,0) 	node (b) {2}
 	(1,0) 		node (c) {\phantom{1}}
 	(1.5,0) 	node (d) {4}
 	(2,0) 		node (e) {\phantom{1}}
 	(2.5,0) 	node (f) {6}
 	(3,0) 		node (g) {7}
 	(3.5,0) 	node (h) {\phantom{1}};
 \draw (a) -- +(0,0.75) -| (c);
 \draw (b) -- +(0,0.60) -| (b);
 \draw (c) -- +(0,0.75) -| (c);
 \draw (d) -- +(0,0.75) -| (e);
 \draw (e) -- +(0,0.75) -| (h);
 \draw (f) -- +(0,0.60) -|(f);
 \draw (g) -- +(0,0.60)-|(g);
 \end{tikzpicture}, 	& \quad t(\sigma) &=&\;\;\begin{tikzpicture}[baseline={([yshift=-1ex]current bounding box.center)},scale=0.6,
level 1/.style={level distance=7mm,sibling distance=7mm}]
\node(0)[solid node,label=right:{\tiny 1}]{}
child{node(1)[solid node,label=right:{\tiny 2}]{}
};
\end{tikzpicture}
\;\; \begin{tikzpicture}[baseline={([yshift=-1ex]current bounding box.center)},scale=0.6,
level 1/.style={level distance=7mm,sibling distance=7mm}]
\node(0)[solid node,label=right:{\tiny 4}]{}
child{node(1)[solid node,label=left:{\tiny 6}]{}
}
child{node(2)[solid node,label=right:{\tiny 7}]{}
};
\end{tikzpicture}
\end{array}
$$
\caption{Noncrossing partitions $\pi$ and $\sigma$ together with their respective forest of nesting $t(\pi)$ and $t(\sigma)$. Each vertex of the forest is decorated with the minimal element of their associated block in the partition.}
\end{figure}
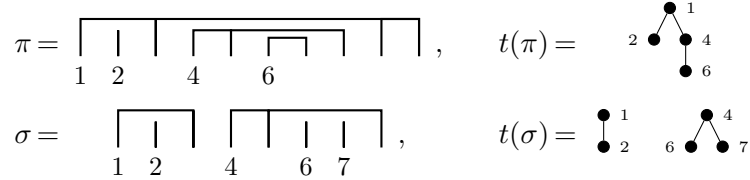

\subsection{Cumulants in noncommutative probability}\label{sec:cumulants}

As stated in the introduction, cumulants are an important tool for dealing with noncommutative notions of independence. Before recalling the definition of cumulants, we fix some notation.
\begin{Definition}
 Let $\A$ be a vector space and let $\{f_n\colon \A^n \to \mathbb{C}\}_{n\geq1}$ be a family of multilinear functionals. For each $\pi \in \NC(n)$ and elements $a_1,\dots,a_n \in \A$, we denote{\samepage
\[
	f_{\pi}(a_1,\dots,a_n) := \prod_{V \in \pi} f_{|V|}(a_1,\dots,a_n | V),
\]
where for $V = \{i_1 < \cdots < i_s\}$, $f_{|V|}(a_1,\dots,a_n|V) := f_s(a_{i_1},\dots,a_{i_s})$.}
\end{Definition}
A \textit{noncommutative probability space} is a pair $(\A,\varphi)$ where $\A$ is a unital algebra over $\mathbb{C}$ and $\varphi\colon \A\to \mathbb{C}$ is a unital linear functional such that $\varphi(1_\A)=1$, where $1_\A$ is the unit of $\A$. In this note, we are not assuming that $\A$ or $\varphi$ have any additional structure, for instance, the existence of an involution on $\A$ or positivity of $\varphi$.
Let $(\A,\varphi)$ be a noncommutative probability space. Associated to it, we can construct a family of multilinear functionals $\{\varphi_n\colon \A^n\to\mathbb{C}\}_{n\geq 1}$ by the defining $\varphi_n(a_1,\dots,a_n) := \varphi(a_1\cdots a_n)$, for any $n\geq1$ and $a_1,\dots,a_n\in \A$. Note that $a_1\cdots a_n$ is a product in $\A$. The definition of cumulants can be stated as follows.

\begin{Definition}
Let $(\A,\varphi)$ be a noncommutative probability space.
\begin{itemize}\itemsep=0pt
\item \textit{Free cumulants} are the family of multilinear functionals $\{r_n\colon \A^n\to\mathbb{C}\}_{n\geq1}$ recursively defined by~\cite{Spe94}
\begin{equation}\label{eq:freeMC}
\varphi_n(a_1,\dots,a_n) = \sum_{\pi\in \NC(n)} r_\pi(a_1,\dots,a_n).
\end{equation}
\item \textit{Boolean cumulants} are the family of multilinear functionals $\{b_n\colon \A^n\to\mathbb{C}\}_{n\geq1}$ recursively defined by~\cite{SW97}
\begin{equation}\label{eq:BoolMC}
\varphi_n(a_1,\dots,a_n) = \sum_{\pi\in \Int(n)} b_\pi(a_1,\dots,a_n).
\end{equation}
\item \textit{Monotone cumulants} are the family of multilinear functionals $\{h_n\colon \A^n\to\mathbb{C}\}_{n\geq1}$ recursively defined by~\cite{HS11b}
\begin{equation}\label{eq:monMC}
\varphi_n(a_1,\dots,a_n) = \sum_{\pi\in \NC(n)} \frac{1}{t(\pi)!} h_\pi(a_1,\dots,a_n).
\end{equation}
\end{itemize}

\end{Definition}
Formulas \eqref{eq:freeMC}, \eqref{eq:BoolMC} and \eqref{eq:monMC} are called the \textit{free, Boolean and monotone moment-cumulant relations}, respectively.

\if

\else

Recall that for each $n\geq1$, $\NC(n)$ is a poset with respect to the reverse refinement order. Also, the restriction of this order to interval partitions gives us a poset structure on $\Int(n)$. Hence, it is possible to invert equations \eqref{eq:freeMC} and \eqref{eq:BoolMC} via M\"obius inversion in order to write the free and Boolean cumulants in terms of moments
\begin{eqnarray}
	\label{eq:freeCM}
	r_n(a_1,\dots,a_n)
	&=& \sum_{\sigma \in \NC(n)} \operatorname{M\ddot{o}b}(\sigma,1_n)\varphi_\sigma(a_1,\dots,a_n),
	\\ \label{eq:BoolCM}
	b_n(a_1,\dots,a_n)
	&=& \sum_{\sigma \in \Int(n)}(-1)^{|\sigma|-1}\varphi_\sigma(a_1,\dots,a_n),
\end{eqnarray}
where $\mathrm{M\ddot{o}b}$ stands for the M\"obius function on the incidence algebra of noncrossing partitions.
On the other hand, by using \eqref{eq:labeling} it is readily shown that the monotone moment-cumulant formula can be written as
\begin{equation}
\label{eq:monMC2}
	\varphi_n(a_1,\dots,a_n) = \sum_{\pi\in \NC(n)} \frac{1}{t(\pi)!} h_\pi(a_1,\dots,a_n),
\end{equation}
for any $a_1,\dots,a_n\in \A$. Unlike the free and Boolean cases, the coefficients $\frac{1}{t(\pi)!}$ are not multiplicative, so we cannot apply M\"obius inversion in order to express the monotone cumulants in terms of moments. The two subsequent sections of this note are devoted to describing the combinatorial tools used with the purpose of obtaining a formula analogous to \eqref{eq:freeCM} and \eqref{eq:BoolCM} for the monotone case.

\fi

\section{Shuffle approach for noncommutative probability}
\label{sec:shuffle}

In this section, we briefly describe the Hopf-algebraic framework of Ebrahimi-Fard and Patras developed in a series of recent articles \cite{EFP1, EFP2, EFP3, EFP4, EFP5}. The key ingredient in the group-theoretical framework of Ebrahimi-Fard and Patras is an unshuffle (or codendriform) structure on the double tensor algebra defined over $\A$. More precisely, given a noncommutative probability space~$(\A,\varphi)$, we consider the double tensor algebra
\[
	T(T_+(\A)) = \bigoplus_{n\geq0} T_+(\A)^{\otimes n},
	\qquad\text{where}\quad
	T_+(\A) = \bigoplus_{n\geq1} \A^{\otimes n}.
\]
In the following, we will omit the tensor symbol when denoting elements of $T_+(\A)$, that is, we consider them as words $w = a_1\cdots a_n \in \A^{\otimes n}$, where $a_1,\dots,a_n\in \A$. On the other hand, the product of words in $T(T_+(\A))$ is denoted by the bar notation, $w_1|\cdots |w_m$, where $w_1,\dots,w_m \allowbreak\in T_+(\A)$. Observe that the empty word, denoted by~$\uno$, works as the unit for the bar product. Moreover, $T(T_+(\A))$ has a multigrading and grading given by
\[
	T(T_+(\A))_{n_1,\dots,n_k}
	:= T_+(\A)^{\otimes n_1} \otimes \cdots \otimes T_+(\A)^{\otimes n_k},
\]
respectively,
\[
	T(T_+(\A))_n :=
	\bigoplus_{n_1+\cdots + n_k = n} T(T_+(\A))_{n_1,\dots,n_k},
\]
for any $n,n_1,\dots,n_k\geq1$. Additionally, we consider the algebra morphism $\epsilon\colon T(T_+(\A))\to\mathbb{C}$ defined to be zero on the augmentation ideal $T_+(T_+(\A)) = \bigoplus_{n\geq 1} T_+(\A)^{\otimes n}$ and $\epsilon(\uno) = 1$.

In order to describe the desired Hopf algebra structure on $T(T_+(\A))$, we have to present the coproduct. To this end,
 we describe the following notation. Given a subset $S\subseteq [n] :=\{1,\dots, n\}$, we define the connected components of $[n]\backslash S$ as the set of maximal intervals contained in $[n]\backslash S$. Now, given a word $w= a_1\cdots a_n$ and a subset $S = \{s_1 <\cdots < s_r\} \subseteq [n]$, we define
\[
	a_S = a_{s_1}\cdots a_{s_r},
\]
with the convention that $a_\emptyset = \uno$ (the unit of $T(T_+(\A))$). Also, if $J_1,\dots,J_r$ be the connected components of $[n] \backslash S$ ordered according to their minimum element, we denote
\[
	a_{J_{[n]}^S} = a_{J_1} | \cdots | a_{J_r}.
\]
With the above notations, we can state the desired coproduct: let $\Delta\colon T(T_+(\A))\to T(T_+(\A))\otimes T(T_+(\A))$ the algebra morphism defined by $\Delta(\uno) = \uno\otimes \uno$ and
\begin{equation}
\label{eq:copEFP}
	\Delta(a_1\cdots a_n)= \sum_{S\subseteq [n]} a_S \otimes a_{J_{[n]}^S}.
\end{equation}

\begin{Example}
For $n=4$, we have
\begin{align*}
	\Delta(a_1a_2a_3a_4)
={} & a_1a_2a_3a_4\otimes \uno + a_1 \otimes a_2a_3a_4 + a_2\otimes a_1| a_3a_4 + a_3 \otimes a_1a_2| a_4 + a_4\otimes a_1a_2a_3
	\\&\!{} + a_1a_2\otimes a_3a_4 + a_1a_3\otimes a_2|a_4 + a_1a_4\otimes a_2a_3 + a_2a_3\otimes a_1|a_4 + a_2a_4\otimes a_1|a_3
	\\&\!{} + a_3a_4\otimes a_1a_2 + a_1a_2a_3\otimes a_4 + a_1a_2a_4\otimes a_3 + a_1a_3a_4\otimes a_2 + a_2a_3a_4\otimes a_1
	\\&\!{} + \uno \otimes a_1a_2a_3a_4.
\end{align*}
\end{Example}
With the above definitions, we have the following result.

\begin{Theorem}[\cite{EFP1}]\label{thm:HopfAlg}
$(T(T_+(\A)),\Delta,\epsilon))$ is a connected graded noncommutative and noncocommutative Hopf algebra.
\end{Theorem}

The fundamental observation of the authors of \cite{EFP1} is that the coproduct defined in \eqref{eq:copEFP} can be split into two parts of the form
\[
	\Delta = \Delta_\prec + \Delta_\succ,
\]
where
\begin{gather*}
	\Delta_\prec (a_1\cdots a_n)
	= \sum_{1\in S\subseteq [n]} a_S \otimes a_{J_{[n]}^S},\qquad
 	\Delta_\succ (a_1\cdots a_n)
	= \sum_{1\not\in S\subseteq [n]} a_S \otimes a_{J_{[n]}^S}
\end{gather*}
for any word $a_1\cdots a_n\in \A^{\otimes n}$, $n\geq 1$. Both maps can be extended to $T(T_+(\A))$ by stating that
\begin{align*}
	\Delta_\prec(w_1|w_2|\cdots |w_n)
	&= \Delta_\prec(w_1)|\Delta(w_2|\cdots|w_n),\\
	\Delta_\succ(w_1|w_2|\cdots |w_n)
	&= \Delta_\succ(w_1)|\Delta(w_2|\cdots|w_n),
\end{align*}
for any words $w_1,\dots,w_n\in T_+(\A)$.

\begin{Theorem}[\cite{EFP1}]
\label{thm:unshuffle}
$(T(T_+(\A)),\Delta_\prec,\Delta_\succ)$ is a unital unshuffle bialgebra.
\end{Theorem}

Unshuffle bialgebras are also called \textit{codendriform bialgebras} \cite{Foissy}. For details, including definitions and a proof of the above theorem, we refer the reader to \cite[Theorem~5]{EFP1}.

Theorem \ref{thm:unshuffle} permits to show that the algebraic dual of $(T(T_+(\A)),\Delta,\epsilon))$ is a unital shuffle algebra. More precisely, for $f,g\in \operatorname{Lin}(T(T_+(\A)),\mathbb{C})$, we can define the associative convolution product dual to~\eqref{eq:copEFP} by
\begin{equation}\label{eq:convol}
	f*g = m_{\mathbb{C}}\circ (f\otimes g)\circ \Delta,
\end{equation}
where $m_{\mathbb{C}}$ stands for the product on $\mathbb{C}$. The above product has the counit $\epsilon$ as its unit. Since we have a splitting of the coproduct $\Delta$, we can also consider the respective dual half-shuffle products associated to the half-unshuffle coproducts:
\begin{gather*}
	f\prec g
	= m_{\mathbb{C}}\circ (f\otimes g)\circ \Delta_{\prec},\qquad
	f\succ g
	= m_{\mathbb{C}}\circ (f\otimes g)\circ \Delta_{\succ}.
\end{gather*}
Neither of the above products is associative. However, they satisfy the noncommutative \textit{shuffle identities}
\begin{gather*}
	(f\prec g)\prec h =f\prec (g*h),\qquad\!
	(f\succ g)\prec h = f\succ (g\prec h),\qquad\!
	f\succ (g\succ h) = (f*g)\succ h,
\end{gather*}
for any $f,g,h\in \operatorname{Lin}(T(T_+(\A)),\mathbb{C})$. In other words, we have:
\begin{Proposition}[\cite{EFP1}]
$(\operatorname{Lin}(T(T_+(\A))),\prec,\succ)$ is a unital noncommutative shuffle algebra.
\end{Proposition}

Recall that a map $\Phi \in \operatorname{Lin}(T(T_+(\A))$ is called a \textit{character} if $\Phi(\uno) = 1$ and $\Phi(w_1|w_2) = \Phi(w_1)\Phi(w_2)$ for any $w_1,w_2\in T(T_+(\A))$. Additionally, a map $\alpha \in \operatorname{Lin}(T(T_+(\A))$ is called an \textit{infinitesimal character} if $\alpha({\bf{1}}) = 0$ and $\alpha(w_1|\cdots| w_k) = 0$ for any $k\geq2$ and any $w_1,\dots,w_k \allowbreak \in T_+(A)$. It is possible to show that the set $G$ of characters on any Hopf algebra forms a group with respect to the convolution product \eqref{eq:convol}. On the other hand, the set $\g$ of infinitesimal characters on any Hopf algebra forms a Lie algebra with respect to the Lie bracket $[\gamma_1,\gamma_2] = \gamma_1*\gamma_2 - \gamma_2*\gamma_1$. From the general theory of Hopf algebras, the exponential map
\[
	 \g \ni \alpha \mapsto \sum_{n=0}^\infty \frac{\alpha^{*n}}{n!}\in G
 \]
is a bijection between $\g$ and $G$.

Recall that we have two more products, the half-shuffles $\prec$ and $\succ$, in the Hopf algebra $\operatorname{Lin}(T(T_+(\A)),\mathbb{C})$. Thus we can consider the left and right half-shuffle exponential maps~$\Expl$ and~$\Expr$ given by
\begin{gather*}
	\Expl(\alpha)
	= \sum_{n\geq0} \alpha^{\prec n},\qquad
	\Expr(\alpha)
	= \sum_{n\geq0} \alpha^{\succ n},
\end{gather*}
for $\alpha\in \operatorname{Lin}(T(T_+(\A)),\mathbb{C})$, where $\alpha^{\prec n} = \alpha \prec (\alpha^{\prec n-1})$, for $n\geq1$, and $\alpha^{\prec 0} = \uno$, and in an analogous way for $\alpha^{\succ n}$. It turns out that both half-shuffle exponentials also provide bijections between $\g$ and $G$. The left and right half-shuffle logarithms, i.e., the inverse maps of the half-shuffle exponentials, are denoted $\mathcal{L}_\prec\colon G \to \g$ respectively $\mathcal{L}_\succ\colon G \to \g$. The content of the main results in \cite{EFP1,EFP3} can be stated in the following theorem.

\begin{Theorem}[\cite{EFP1,EFP3}]\label{thm:basicEFP}
For $\Phi$ a character, there exists a unique triple $(\kappa,\beta,\rho)$ of infinitesimal characters such that
\begin{equation}\label{eq:shuffleMCRel}
	\Phi = \exp^*(\rho) = \Expl(\kappa) = \Expr(\beta).
\end{equation}
In particular, we have that $\kappa$ and $\beta$ are the unique solutions of the half-shuffle fixed point equations
\begin{equation}\label{eq:freeEq}
	\Phi = \epsilon + \kappa\prec \Phi,
\end{equation}
respectively
\begin{equation}\label{eq:booleanEq}
	\Phi = \epsilon +\Phi\succ \beta.
\end{equation}
Conversely, given $\alpha\in \g$, then $\exp^*(\alpha)$, $\Expl(\alpha)$ and $\Expr(\alpha)$ are characters.
\end{Theorem}

From \eqref{eq:freeEq} and \eqref{eq:booleanEq} one deduces that the half-shuffle logarithms satisfy the shuffle equations $\mathcal{L}_\prec(\Phi)=(\Phi-\epsilon) \prec\Phi^{*-1}=\kappa$ respectively $\mathcal{L}_\succ(\Phi)=\Phi^{*-1} \succ (\Phi-\epsilon)=\beta$.

 With the above machinery, we can effectively describe the moment-cumulant relations in noncommutative probability in an algebraic framework. To this end, consider a~noncommutative probability space $(\A,\varphi)$ and the double tensor algebra $T(T_+(\A))$. The linear functional $\varphi$ can be lifted to a character $\Phi$ on $T(T_+(\A))$ in the following way: if $w = a_1\cdots a_n \in T_+(\A)$, we define $\Phi(w) = \varphi(a_1\cdots a_n)$ (recall that the product in the argument of~$\varphi$ is the product of the algebra~$\A$) and extend this definition linearly and multiplicatively to a general element in~$T(T_+(\A))$. In general, given a family of multilinear functionals $\{f_n\colon \A^n\to \mathbb{C} \}_{n\geq1}$, we define the infinitesimal character~$\alpha \in g$ associated to the family $\{f_n\}_{n\geq1}$ by $\alpha(w) := f_n(a_1,\dots,a_n)$ if $w= a_1\cdots a_n$. The next theorem states the link between noncommutative probability and the three bijections appearing in Theorem~\ref{thm:basicEFP}.

\begin{Theorem}[\cite{EFP1,EFP3}]\label{thm:linkNCP}
Let $(\A,\varphi)$ be a noncommutative probability space and let $\Phi\colon T(T_+(\A)) \allowbreak \to \mathbb{C}$ the extension of $\varphi$ to a character as described above. Let $(\kappa,\beta,\rho)$ be the unique triple of infinitesimal characters given in Theorem~{\rm \ref{thm:basicEFP}}. Then for every word $w = a_1\cdots a_n \in T_+(\A)$, we have that
\[
	\kappa(w) = r_n(a_1,\dots,a_n),
	\qquad
	\beta(w) = b_n(a_1,\dots,a_n),
	\qquad
	\rho(w) = h_n(a_1,\dots,a_n).
\]
In other words, $\kappa$, $\beta$ and $\rho$ are the infinitesimal characters on $T(T_+(\A))$ associated to the free, Boolean, and monotone cumulant functionals, respectively.
\end{Theorem}
The previous theorem allows us to recover the moment-cumulant formulas in terms of noncrossing partitions when we evaluate the equations in \eqref{eq:shuffleMCRel} in a word $w= a_1\cdots a_n\in T_+(\A)$:
\begin{gather*}
	\Expl(\kappa)(a_1\cdots a_n)
	= \sum_{\pi\in \NC(n)} r_\pi(a_1,\dots,a_n),\\
	\Expr(\beta)(a_1\cdots a_n)
	= \sum_{\pi\in \Int(n)} b_\pi(a_1,\dots,a_n),\\
	\exp^*(\rho)(a_1\cdots a_n)
	= \sum_{\pi\in \NC(n)} \frac{1}{t(\pi)!}h_\pi(a_1,\dots,a_n).
\end{gather*}

\section{A Hopf algebra of Schr\"oder trees}\label{sec:schroder}

The second combinatorial object needed to explain our desired formula for expressing monotone cumulants in terms of moments is a special case of planar rooted trees along with a rich associated algebraic structure. The objective of this section is to describe a Hopf algebra on such rooted trees and a coalgebra morphism that will allow lifting the problem of computing the shuffle logarithm of the character $\Phi$ associated to a noncommutative probability space~$(\A,\varphi)$.

\begin{Definition}A \textit{Schr\"oder tree} is a planar rooted tree such that each of its internal vertices has at least two descendants. For any $n\geq0$, we will denote by $\ST(n)$ the set of Schr\"oder trees with $n+1$ leaves.
\end{Definition}
\begin{Example}
The following list contains the elements of $\ST(3)$:

\begin{figure}[h]\centering
\begin{tabular}{c c c c c c }
\begin{tikzpicture}[scale=0.7,
level 1/.style={level distance=7mm,sibling distance=7mm}]
\node(0)[solid node,label=above:{}]{}
child{node(1)[hollow node]{}
}
child{node(2)[hollow node]{}
}
child{[black] node(11)[hollow node]{}
}
child{[black] node(111)[hollow node]{}
};
\end{tikzpicture} & \begin{tikzpicture}[scale=0.7,
level 1/.style={level distance=7mm,sibling distance=7mm}]
\node(0)[solid node,label=above:{}]{}
child{node(1)[hollow node]{}
}
child{node(2)[solid node]{}
	child{[black] node(11)[hollow node]{}}
	child{[black] node(111)[hollow node]{}}
	child{[black] node(1111)[hollow node]{}}
};
\end{tikzpicture} &
\begin{tikzpicture}[scale=0.7,
level 1/.style={level distance=7mm,sibling distance=7mm}]
\node(0)[solid node,label=above:{}]{}
child{node(1)[solid node]{}
	child{[black] node(11)[hollow node]{}}
	child{[black] node(111)[hollow node]{}}
}
child{node(2)[hollow node]{}
}
child{node(3)[hollow node]{}};
\end{tikzpicture} & \begin{tikzpicture}[scale=0.7,
level 1/.style={level distance=7mm,sibling distance=7mm}]
\node(0)[solid node,label=above:{}]{}
child{node(1)[solid node]{}
	child{[black] node(11)[hollow node]{}}
	child{[black] node(111)[hollow node]{}}
	child{[black] node(1111)[hollow node]{}}
}
child{node(2)[hollow node]{}
};
\end{tikzpicture} & \begin{tikzpicture}[scale=0.7,
level 1/.style={level distance=7mm,sibling distance=7mm}]
\node(0)[solid node,label=above:{}]{}
child{node(1)[hollow node]{}
}
child{node(2)[hollow node]{}
}
child{node(3)[solid node]{}
	child{[black] node(11)[hollow node]{}}
	child{[black] node(111)[hollow node]{}}};
\end{tikzpicture} &\begin{tikzpicture}[scale=0.7,
level 1/.style={level distance=7mm,sibling distance=7mm}]
\node(0)[solid node,label=above:{}]{}
child{node(1)[hollow node]{}
	}
child{node(2)[solid node]{}
	child{[black] node(11)[hollow node]{}}
	child{[black] node(111)[hollow node]{}}
}
child{node(3)[hollow node]{}
};
\end{tikzpicture} \\ a)& b)& c)& d)& e) & f) \\ \begin{tikzpicture}[scale=0.7,
level 1/.style={level distance=7mm,sibling distance=7mm}]
\node(0)[solid node,label=above:{}]{}
child{node(1)[solid node]{}
 	child{[black] node(11)[solid node]{}
		child{[black] node(112)[hollow node]{}}
		child{[black] node(1112)[hollow node]{}}
	}
	child{[black] node(111)[hollow node]{}}
	}
child{node(2)[hollow node]{}
};
\end{tikzpicture} & \begin{tikzpicture}[scale=0.7,
level 1/.style={level distance=7mm,sibling distance=7mm}]
\node(0)[solid node,label=above:{}]{}
child{node(1)[solid node]{}
 	child{[black] node(11)[hollow node]{}
	}
	child{[black] node(111)[solid node]{}
		child{[black] node(112)[hollow node]{}}
		child{[black] node(1112)[hollow node]{}}
	}	
	}
child{node(2)[hollow node]{}
};
\end{tikzpicture} & \begin{tikzpicture}[scale=0.7,
level 1/.style={level distance=7mm,sibling distance=7mm}]
\node(0)[solid node,label=above:{}]{}
child{node(1)[hollow node]{}	
	}
child{node(2)[solid node]{}
	child{[black] node(11)[solid node]{}
		child{[black] node(112)[hollow node]{}}
		child{[black] node(1112)[hollow node]{}}
	}
	child{[black] node(111)[hollow node]{}	}
};
\end{tikzpicture} & \begin{tikzpicture}[scale=0.7,
level 1/.style={level distance=7mm,sibling distance=7mm}]
\node(0)[solid node,label=above:{}]{}
child{node(1)[hollow node]{}
	}
child{node(2)[solid node]{}
	child{[black] node(11)[hollow node]{}
	}
	child{[black] node(111)[solid node]{}
		child{[black] node(112)[hollow node]{}}
		child{[black] node(1112)[hollow node]{}}
	}
};
\end{tikzpicture} & \begin{tikzpicture}[scale=0.7,
level 1/.style={level distance=8mm,sibling distance=12mm},
level 2/.style={level distance=8mm,sibling distance=6mm},
]
\node(0)[solid node,label=above:{}]{}
child{node(1)[solid node]{}
	child{[black] node(112)[hollow node]{}}
		child{[black] node(1112)[hollow node]{}}
	}
child{node(2)[solid node]{}
	child{[black] node(11)[hollow node]{}
	}
	child{[black] node(111)[hollow node]{}
	}
};
\end{tikzpicture} & \\ g)& h) &i) &j) &k) &
\end{tabular}
\vspace{0.5cm}
\caption{Schr\"oder trees with four leaves. Internal vertices are black-coloured, while leaves are white-coloured.}
\label{fig:st3}
\end{figure}
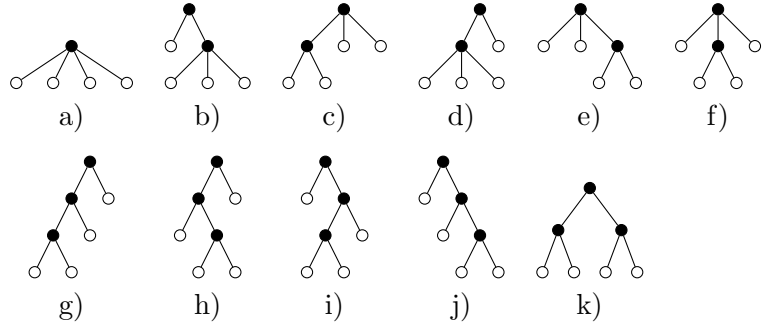
\end{Example}

Schr\"oder trees were used extensively in reference \cite{JMNT} in order to describe the functional equation between the moment series and the $R$-transform of noncommutative random variables from an operadic point of view. The authors of \cite{JMNT} also introduced a Hopf algebra related to their operad of Schr\"oder trees and found relations with the works \cite{EFP1,EFP3}. The purpose of the paper \cite{JMNT} is to understand the free cumulants as linear functionals on an unshuffle bialgebra more general than the double tensor algebra.

The description of the aforementioned Hopf algebra is as follows: let $\cHS$ be the noncommutative polynomial algebra whose indeterminates are given by the elements of the set of Schr\"oder trees $\ST = \bigcup_{n\geq0} \ST(n)$. More precisely
\[
	\cHS = \mathbb{C}\langle t\colon t\in \ST\rangle/ (\circ -1),
\]
where we identify $\circ$, the unique element of $\ST(0)$, with the complex unit $1$.

In order to describe the coproduct on Schr\"oder trees, we have to consider the notion of {admissible cut} on a tree $t$.

\begin{Definition}
\label{def:admcut}
Let $t$ be a Schr\"oder tree. An \textit{admissible cut on $t$} is a subset $c$ of internal vertices of $t$ such that for any path from the root to any leaf, there is at most one vertex of the path contained in $c$.
\end{Definition}

Given an admissible cut $c$, it is clear that the elements in $c$ can be naturally ordered from left to right according to their position in the Schr\"oder tree. This natural order allows us to give a~well-defined coproduct.

\begin{Definition}\label{def:PrunCuts}
Let $t$ be a Schr\"oder tree and $c$ be an admissible cut of $t$.
\begin{enumerate}\itemsep=0pt
\item[i)] We define the \textit{pruning of $t$} associated to $c$ as the ordered forest of Schr\"oder trees $P_c(t)$ formed by the subtrees obtained by cutting the edge above each element of $c$.
\item[ii)] We define the \textit{trunk of $t$} associated to $c$ as the Schr\"oder tree $R_c(t)$ obtained by replacing in $t$ each subtree of $P_c(t)$ by a leaf.
\end{enumerate}
\end{Definition}

\begin{Remark}In the context of the previous definition, we should notice that if $c$ is the admissible cut only containing the root of a Schr\"oder tree~$t$, then we set~$R_c(t) = \circ$ and $P_c(t) = t$.
\end{Remark}

\begin{Remark}Observe that given an admissible cut $c$ of a Schr\"oder tree $t$, the trunk, $R_c(t)$, is a~Schr\"oder tree and the ordered forest, $P_c(t)$, can be seen as a noncommutative monomial in~$\cHS$. Also, each element in $P_c(t)$ has a unique vertex in $c$ as its root.
\end{Remark}

With the above definition, the coproduct on $\cHS$ can be defined as the algebra morphism $\delta\colon \cHS\to \cHS\otimes \cHS$ given by
\begin{equation}\label{eq:coproductH}
	\delta(t) = \sum_{c \in \operatorname{Adm}(t)} R_c(t)\otimes P_c(t),
\end{equation}
for any Schr\"oder tree $t \in \ST$, where $\operatorname{Adm}(t)$ stands for the set of admissible cuts of~$t$.

\begin{Remark}
One should notice that Definitions \ref{def:admcut} and \ref{def:PrunCuts} can be directly extended to the case of Schr\"oder forests, i.e., noncommutative monomials of Schr\"oder trees. The multiplicative of the coproduct $\delta$ is then represented by the fact that if $f$ is a Schr\"oder forest, then $\delta(f)$ is defined with the same expression that of~\eqref{eq:coproductH} but now considering the corresponding definition of trunk and pruning of a Schr\"oder forest $f$.
\end{Remark}

\begin{Example}\label{ex:schrodertree}
We have the following simple computations of the coproduct $\delta$:
\[
\delta\bigg( \begin{tikzpicture}[baseline={([yshift=-1.4ex]current bounding box.center)},scale=0.4,
level 1/.style={level distance=8mm,sibling distance=12mm},
level 2/.style={level distance=8mm,sibling distance=6mm},
]
\node(0)[solid node,label=above:{}]{}
child{node(1)[solid node]{}
	child{[black] node(112)[hollow node]{}}
		child{[black] node(1112)[hollow node]{}}
	}
child{node(2)[solid node]{}
	child{[black] node(11)[hollow node]{}
	}
	child{[black] node(111)[hollow node]{}
	}
};
\end{tikzpicture} \bigg) = \begin{tikzpicture}[scale=0.4,
level 1/.style={level distance=7mm,sibling distance=7mm}]
\node(0)[hollow node,label=above:{}]{};\end{tikzpicture} \otimes \begin{tikzpicture}[scale=0.4,
level 1/.style={level distance=8mm,sibling distance=12mm},
level 2/.style={level distance=8mm,sibling distance=6mm},
]
\node(0)[solid node,label=above:{}]{}
child{node(1)[solid node]{}
	child{[black] node(112)[hollow node]{}}
		child{[black] node(1112)[hollow node]{}}
	}
child{node(2)[solid node]{}
	child{[black] node(11)[hollow node]{}
	}
	child{[black] node(111)[hollow node]{}
	}
};
\end{tikzpicture} +\begin{tikzpicture}[scale=0.4,
level 1/.style={level distance=8mm,sibling distance=12mm},
level 2/.style={level distance=8mm,sibling distance=6mm},
]
\node(0)[solid node,label=above:{}]{}
child{node(1)[solid node]{}
	child{[black] node(112)[hollow node]{}}
		child{[black] node(1112)[hollow node]{}}
	}
child{node(2)[solid node]{}
	child{[black] node(11)[hollow node]{}
	}
	child{[black] node(111)[hollow node]{}
	}
};
\end{tikzpicture} \otimes \begin{tikzpicture}[scale=0.4,
level 1/.style={level distance=7mm,sibling distance=7mm}]
\node(0)[hollow node,label=above:{}]{};\end{tikzpicture} + \begin{tikzpicture}[scale=0.4,
level 1/.style={level distance=8mm,sibling distance=8mm},
]
\node(0)[solid node,label=above:{}]{}
child{node(1)[hollow node]{}
	}
child{node(2)[solid node]{}
child{[black] node(11)[hollow node]{}
	}
	child{[black] node(111)[hollow node]{}
	}
};
\end{tikzpicture}\otimes \begin{tikzpicture}[scale=0.4,
level 1/.style={level distance=8mm,sibling distance=8mm},
]
\node(0)[solid node,label=above:{}]{}
child{node(1)[hollow node]{}
	}
child{node(2)[hollow node]{}
};
\end{tikzpicture} + \begin{tikzpicture}[scale=0.4,
level 1/.style={level distance=8mm,sibling distance=8mm},
]
\node(0)[solid node,label=above:{}]{}
child{node(1)[solid node]{}
	child{[black] node(11)[hollow node]{}
	}
	child{[black] node(111)[hollow node]{}
	}
	}
child{node(2)[hollow node]{}
};
\end{tikzpicture}\otimes \begin{tikzpicture}[scale=0.4,
level 1/.style={level distance=8mm,sibling distance=8mm},
]
\node(0)[solid node,label=above:{}]{}
child{node(1)[hollow node]{}
	}
child{node(2)[hollow node]{}
};
\end{tikzpicture} + \begin{tikzpicture}[scale=0.4,
level 1/.style={level distance=8mm,sibling distance=8mm},
]
\node(0)[solid node,label=above:{}]{}
child{node(1)[hollow node]{}
	}
child{node(2)[hollow node]{}
};
\end{tikzpicture} \otimes \begin{tikzpicture}[scale=0.4,
level 1/.style={level distance=8mm,sibling distance=8mm},
]
\node(0)[solid node,label=above:{}]{}
child{node(1)[hollow node]{}
	}
child{node(2)[hollow node]{}
};
\end{tikzpicture} \;\begin{tikzpicture}[scale=0.4,
level 1/.style={level distance=8mm,sibling distance=8mm},
]
\node(0)[solid node,label=above:{}]{}
child{node(1)[hollow node]{}
	}
child{node(2)[hollow node]{}
};
\end{tikzpicture}\; .
\]
\end{Example}

For a tree $t\in \ST(n)$, we define the \textit{degree of t} by $\operatorname{deg}(t) = n$. More generally, for an ordered forest of Schr\"oder trees $f = t_1\cdots t_m$, we define
\[
	\operatorname{deg}(f) = \operatorname{deg}(t_1) + \cdots + \operatorname{deg}(t_m).
\]
Furthermore, the algebra $\cHS$ is graded, with the component $\cHS(n)$ of degree $n$ given by the linear span of all ordered forests of degree $n$ for every $n\geq0$.

Finally, for the counit, we consider the algebra morphism $\epsilon\colon \cHS\to\mathbb{C}$ given by
\[
	\varepsilon(t) = \begin{cases}1 & \text{if $t=\circ$,}\\ 0 &\text{if $t\in \cHS(n)$ for $n\geq1$.}\end{cases}
\]
With the above ingredients, we have the following result.

\begin{Proposition}[\cite{JMNT}]
$(\cHS,\delta,\varepsilon)$ is a connected graded Hopf algebra.
\end{Proposition}

For our purposes of computing multivariate moments and cumulants of random variables, the above Hopf algebra can be enhanced in order to consider decorations. Indeed, let $\A$ be vector space over $\mathbb{C}$ and consider the tensor algebra $T(\A) = \bigoplus_{n\geq0} \A^{\otimes n}$. Define the algebra given by
\[
	\cHS(\A) = \bigoplus_{n\geq0} \big(\cHS(n) \otimes \A^{\otimes n}\big),
\]
with the product being inherited by the algebra structures of $\cHS$ and $T(\A)$. Now, consider an element $t\otimes w\in \cHS(n) \otimes \A^{\otimes n}$ where $w = a_1\cdots a_n$; here we omit the tensor notation of the elements in $\A^{\otimes n}$ in the same way we did in Section~\ref{sec:shuffle}. This pure tensor element can be identified with a~decorated Schr\"oder tree, where we label the $n$ sectors between the $n+1$ leaves of $t$ from left to right with $a_1,\dots, a_n$.

We will associate a noncrossing partition to any decorated Schr\"oder tree as follows.

\begin{Definition}\label{def:pit}
Let $t$ be a Schr\"oder tree with $n+1$ leaves. We label the sectors between the leaves from left to right from $1$ to $n$. We will associate a noncrossing partition $\pi(t) \in \NC(n)$ as follows: for each internal vertex~$v$ of~$t$, consider the corolla that has as root $v$ and is a subtree of $t$. For this corolla, define the block $V$ whose elements are the labels of the sectors that define the corolla.
\end{Definition}

\begin{figure}[h]\centering
\begin{minipage}[b]{.4\textwidth}
 \centering
 \begin{tikzpicture}[scale=1.5,font=\footnotesize,
 level 1/.style={level distance=5mm,sibling distance=9mm},
 level 2/.style={level distance=5mm,sibling distance=6mm},
 level 3/.style={level distance=6mm,sibling distance=5mm},
 level 4/.style={level distance=10mm,sibling distance=5mm},
]
\node(0)[solid node,label=above:{}]{}
child{node(1)[hollow node]{}
}
child{node(2)[solid node]{}
child{[black] node(11)[hollow node]{}}
child{[black] node(12)[hollow node]{}}
child{[black] node(13)[solid node]{}
	child{[black] node(131)[hollow node]{}}
	child{[black] node(132)[hollow node]{}}
}
edge from parent node[left]{}
}
child{node(3)[hollow node]{}
}
child{node(4)[solid node]{}
child{[black] node(41)[hollow node, ]{}}
child{[black] node(42)[solid node,]{}
	child{[black] node(421)[hollow node, ]{}}
	child{[black] node(422)[hollow node,]{}}
	child{[black] node(423)[hollow node, ]{}}
	child{[black] node(424)[hollow node,]{}}
}
edge from parent node[right]{}
};

\path (1) -- node {$a_1$} (2);
\path (2) -- node {$a_5$} (3);
\path (3) -- node {$a_6$} (4);

\path (11) -- node (H) {$a_{2}$} (12);
\path (12) -- node {$a_3$} (13);
\path (131) -- node {$a_4$} (132);
\path (41) -- node {$a_{7}$} (42);
\path (421) -- node {$a_8$} (422);
\path (422) -- node {$a_9$} (423);
\path (423) -- node {$a_{10}$} (424);

\end{tikzpicture}
\end{minipage}%
\begin{minipage}[b]{.4\textwidth}
 \centering
 \begin{tikzpicture}[thick,font=\small]
 \path 	(0,0) 		node (a) {1}
 	(0.5,0) 	node (b) {2}
 	(1,0) 		node (c) {3}
 	(1.5,0) 	node (d) {4}
 	(2,0) 		node (e) {5}
 	(2.5,0) 	node (f) {6}
 	(3,0) 		node (g) {7}
 	(3.5,0) 	node (h) {8}
 	(4,0) 		node (i) {9}
 	(4.5,0) 	node (j) {10};
 \draw (a) -- +(0,0.75) -| (f);
 \draw (e) -- +(0,0.75) -| (e);
 \draw (b) -- +(0,0.60) -| (c);
 \draw (d) -- +(0,0.60) -| (d);
 \draw (g) -- +(0,0.75) -| (g);
 \draw (h) -- +(0,0.75) -| (j);
 \draw (i) -- +(0,0.75) -| (i);
 \end{tikzpicture}
\end{minipage}
\caption{Decorated Schr\"oder tree $t \otimes a_1\cdots a_{10}$ (left). The associated noncrossing partition of $t$ is $\pi(t) = \{\{ 1,5,6\},\{2,3\},\{4\}, \{7\}, \{8,9,10\}\}$ (right). }\label{fig:DecSchTree}
\end{figure}
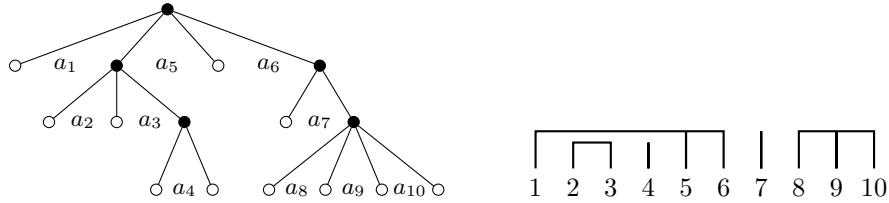

Hence \eqref{eq:coproductH} also defines a coproduct $\delta_\A$ in $\cHS(\A)$ with the additional rule of carrying the decorations of the subtrees obtained in $R_c(t)$ and $P_c(t)$:
\[
	\delta_{\A} (t\otimes w)= \sum_{c\in \operatorname{Adm}(t)} R_c(t\otimes w)\otimes P_c(t\otimes w).
\]

\begin{Example}
Considering the decoration $a_1a_2a_3\in \A^{\otimes 3}$ of the Schr\"oder tree in Example \ref{ex:schrodertree}, we have the following computation of the coproduct in $H_{\mathcal{S}}(\mathcal{A})$:
\begin{align*}
\delta_\A \bigg( \begin{tikzpicture}[baseline={([yshift=-1.4ex]current bounding box.center)},scale=0.6,
level 1/.style={level distance=8mm,sibling distance=24mm},
level 2/.style={level distance=8mm,sibling distance=12mm},
]
\node(0)[solid node,label=above:{}]{}
child{node(1)[solid node]{}
	child{[black] node(112)[hollow node]{}}
		child{[black] node(1112)[hollow node]{}}
	}
child{node(2)[solid node]{}
	child{[black] node(11)[hollow node]{}
	}
	child{[black] node(111)[hollow node]{}
	}
};
\path (112) -- node {$a_1$} (1112);
\path (1112) -- node {$a_2$} (11);
\path (11) -- node {$a_3$} (111);
\end{tikzpicture}
\bigg)={} &\begin{tikzpicture}[scale=0.6,
level 1/.style={level distance=7mm,sibling distance=7mm}]
\node(0)[hollow node,label=above:{}]{};\end{tikzpicture} \otimes \begin{tikzpicture}[baseline={([yshift=-1ex]current bounding box.center)},scale=0.6,
level 1/.style={level distance=8mm,sibling distance=24mm},
level 2/.style={level distance=8mm,sibling distance=12mm},
]
\node(0)[solid node,label=above:{}]{}
child{node(1)[solid node]{}
	child{[black] node(112)[hollow node]{}}
		child{[black] node(1112)[hollow node]{}}
	}
child{node(2)[solid node]{}
	child{[black] node(11)[hollow node]{}
	}
	child{[black] node(111)[hollow node]{}
	}
};
\path (112) -- node {$a_1$} (1112);
\path (1112) -- node {$a_2$} (11);
\path (11) -- node {$a_3$} (111);
\end{tikzpicture} + \begin{tikzpicture}[baseline={([yshift=-1ex]current bounding box.center)},scale=0.6,
level 1/.style={level distance=8mm,sibling distance=24mm},
level 2/.style={level distance=8mm,sibling distance=12mm},
]
\node(0)[solid node,label=above:{}]{}
child{node(1)[solid node]{}
	child{[black] node(112)[hollow node]{}}
		child{[black] node(1112)[hollow node]{}}
	}
child{node(2)[solid node]{}
	child{[black] node(11)[hollow node]{}
	}
	child{[black] node(111)[hollow node]{}
	}
};
\path (112) -- node {$a_1$} (1112);
\path (1112) -- node {$a_2$} (11);
\path (11) -- node {$a_3$} (111);
\end{tikzpicture} \otimes \begin{tikzpicture}[scale=0.6,
level 1/.style={level distance=7mm,sibling distance=7mm}]
\node(0)[hollow node,label=above:{}]{};\end{tikzpicture} + \begin{tikzpicture}[baseline={([yshift=-1ex]current bounding box.center)},scale=0.6,
level 1/.style={level distance=8mm,sibling distance=14mm},
]
\node(0)[solid node,label=above:{}]{}
child{node(1)[hollow node]{}
	}
child{node(2)[solid node]{}
child{[black] node(11)[hollow node]{}
	}
	child{[black] node(111)[hollow node]{}
	}
};
\path (1) -- node {$a_2$} (11);
\path (11) -- node {$a_3$} (111);
\end{tikzpicture}\otimes \begin{tikzpicture}[baseline={([yshift=-1ex]current bounding box.center)},scale=0.6,
level 1/.style={level distance=8mm,sibling distance=12mm},
]
\node(0)[solid node,label=above:{}]{}
child{node(1)[hollow node]{}
	}
child{node(2)[hollow node]{}
};
\path (1) -- node {$a_1$} (2);
\end{tikzpicture} \\ &{} + \begin{tikzpicture}[baseline={([yshift=-1ex]current bounding box.center)},scale=0.6,
level 1/.style={level distance=8mm,sibling distance=14mm},
]
\node(0)[solid node,label=above:{}]{}
child{node(1)[solid node]{}
	child{[black] node(11)[hollow node]{}
	}
	child{[black] node(111)[hollow node]{}
	}
	}
child{node(2)[hollow node]{}
};
\path (11) -- node {$a_1$} (111);
\path (111) -- node {$a_2$} (2);
\end{tikzpicture}\otimes \begin{tikzpicture}[baseline={([yshift=-1ex]current bounding box.center)},scale=0.6,
level 1/.style={level distance=8mm,sibling distance=12mm},
]
\node(0)[solid node,label=above:{}]{}
child{node(1)[hollow node]{}
	}
child{node(2)[hollow node]{}
};
\path (1) -- node {$a_3$} (2);
\end{tikzpicture} + \begin{tikzpicture}[baseline={([yshift=-1ex]current bounding box.center)},scale=0.6,
level 1/.style={level distance=8mm,sibling distance=12mm},
]
\node(0)[solid node,label=above:{}]{}
child{node(1)[hollow node]{}
	}
child{node(2)[hollow node]{}
};
\path (1) -- node {$a_2$} (2);
\end{tikzpicture} \otimes \begin{tikzpicture}[baseline={([yshift=-1ex]current bounding box.center)},scale=0.6,
level 1/.style={level distance=8mm,sibling distance=12mm},
]
\node(0)[solid node,label=above:{}]{}
child{node(1)[hollow node]{}
	}
child{node(2)[hollow node]{}
};
\path (1) -- node {$a_1$} (2);
\end{tikzpicture} \;\begin{tikzpicture}[baseline={([yshift=-1ex]current bounding box.center)},scale=0.6,
level 1/.style={level distance=8mm,sibling distance=12mm},
]
\node(0)[solid node,label=above:{}]{}
child{node(1)[hollow node]{}
	}
child{node(2)[hollow node]{}
};
\path (1) -- node {$a_3$} (2);
\end{tikzpicture} .
\end{align*}
\end{Example}

We also have a grading inherited by $\cHS$ and the counit $\varepsilon_\A\colon \cHS(\A)\to\mathbb{C}$ is provided by the algebra morphism such that
\[
	\varepsilon_\A(t\otimes w) =
		\begin{cases} 1
		& \text{if $t=\circ$ and $w=\uno$,}\\ 0 &\text{otherwise.}
		\end{cases}
\]
Since $\operatorname{\deg}(t) = n$ if $t\in \ST(n)$, i.e., if $t$ has $n+1$ leaves, we obtain again that $t\otimes w$ has degree 0 if and only if $t$ has one leaf and $w$ has length~0, and this happens if only if $t$ is a single-vertex tree decorated with the empty word. Thus, $0$-th degree component of $H_\mathcal{S}(\A) \cong \mathbb{C}$ and hence, $H_\mathcal{S}(\A)$ is connected. Therefore we obtain:

\begin{Theorem}[\cite{JMNT}]
$(\cHS(\A),\delta_{\A},\varepsilon_\A)$ is a connected graded Hopf algebra.
\end{Theorem}

An important feature of the previous Hopf algebra is that its coproduct can be split into two half-coproducts in order to endow $\cHS(\A)$ with an unshuffle bialgebra structure. This splitting is described in the following way: Let $\cHS^+(\A)$ be the augmentation ideal of $\cHS(\A)$ and $t\otimes w \in \cHS^+$ a decorated Schr\"oder tree. We consider the subset of admissible cuts of a Schr\"oder tree given by
\[
	\operatorname{Adm}_\prec(t) := \{c\in \operatorname{Adm}(t)\colon \text{$R_c(t)$ contains the leftmost leaf of $t$}\}.
\]
We define the maps
\[
	\delta^+_{\prec,\A} (t\otimes w)= \sum_{c\in \operatorname{Adm}_\prec(t)} R_c(t\otimes w)\otimes P_c(t\otimes w),
\]
and
\[
	\delta_{\succ,\A}^+(t\otimes w) = \delta_\A(t\otimes w) - \delta_{\prec,\A}^+(t\otimes w).
 \]
These maps extend to $\delta_{\prec,\A},\delta_{\succ,\A} \colon \cHS^+(\A) \to \cHS(\A)\otimes \cHS(\A)$ by
\begin{align*}
	\delta_{\prec,\A}\big((t_1\otimes w_1)\cdots (t_m\otimes w_m)\big)
	&= \delta_{\prec,\A}(t_1\otimes w_1)\delta_\A\big((t_2\otimes w_2)\cdots (t_m\otimes w_m)\big),\\
	\delta_{\succ,\A}\big((t_1\otimes w_1)\cdots (t_m\otimes w_m)\big)
	&= \delta_{\succ,\A}(t_1\otimes w_1)\delta_\A\big((t_2\otimes w_2)\cdots (t_m\otimes w_m)\big).
\end{align*}
The precise statement of the new structure on $\cHS(\A)$ is as follows.

\begin{Theorem}[\cite{JMNT}]
$(\cHS(\A),\delta_{\prec,\A},\delta_{\succ,\A})$ is an unshuffle bialgebra.
\end{Theorem}

\begin{Remark}
In \cite{JMNT}, the authors give the definition of $\operatorname{Adm}_\prec(t)$ considering the rightmost leaf of $t$ instead of the leftmost one. Both structures are clearly isomorphic by considering the reversal of words $a_1\cdots a_n \in \A^{\otimes n}\mapsto a_n\cdots a_1 \in \A^{\otimes n}$.
\end{Remark}

Let $\A$ be an algebra. With the above theorem, we now have two unshuffle bialgebras $T(T_+(\A))$ and $\cHS(\A)$. The relation between both unshuffle bialgebras has been described in~\cite{JMNT} in terms of an algebra morphism which respects both coalgebra and unshuffle bialgebra structures. More precisely:

\begin{Theorem}[\cite{JMNT}]
Let $\iota\colon T(T_+(\A))\to \cHS(\A)$ be the algebra morphism defined by
\begin{equation}\label{eq:iota}
	\iota(w) = \sum_{t\in \ST(n)} t\otimes w,
\end{equation}
for any word $w= a_1\cdots a_n \in T_+(\A)$. Then $\iota$ is a coalgebra morphism and an unshuffle bialgebra morphism.
\end{Theorem}

\section{Monotone cumulants in terms of moments via Schr\"oder trees}\label{sec:monotone}

After having introduced the appropriate tools for our approach, we are prepared to attack the problem of finding a formula to express monotone cumulants in terms of moments.

 We begin with a noncommutative probability space $(\A,\varphi)$. By taking the algebra $\A$, we can consider both Hopf algebras $T(T_+(\A))$ and $\cHS(\A)$ described in Sections \ref{sec:shuffle} and \ref{sec:schroder}. We extend the linear functional $\varphi$ to a character $\Phi\colon T(T_+(\A))\to\mathbb{C}$ as it is done in Theorem \ref{thm:linkNCP}. Observe that we also have the notions of characters and infinitesimal characters on $\cHS(\A)$ defined in a~way analogous to the case of the Hopf algebra $T(T_+(\A))$. Hence we consider the following lifting of~$\varphi$ to a character $\tilde\Phi\colon \cHS(\A)\to\mathbb{C}$ by defining $\tilde\Phi(\circ) =1$ and
\begin{equation}
\label{eq:tildephi}
	\tilde\Phi(t \otimes a_1\cdots a_n)
	= \begin{cases} \varphi(a_1\cdots a_n) & \text{if $t$ is a corolla with $n+1$ leaves},\\ 0& \text{otherwise}. \end{cases}
\end{equation}
By using the coalgebra morphism $\iota$ defined in \eqref{eq:iota}, it is straightforward to see that $\Phi = \tilde\Phi \circ \iota.$ In addition, let $\tilde \rho\colon \cHS(\A)\to\mathbb{C}$ be the infinitesimal character such that $\tilde\Phi = \exp^*(\tilde\rho)$, where in this case $*$ stands for the convolution product defined in $\operatorname{Lin}(\cHS(\A),\mathbb{C})$. As stated in the following result, we can relate $\tilde\rho$ with the monotone cumulants on $\A$.

\begin{Proposition}\label{prop:easy}
Let $\rho\colon T(T_+(\A))\to\mathbb{C}$ be defined by $\rho = \tilde\rho\circ\iota$. Then $\rho$ is a infinitesimal character such that $\Phi = \exp^*(\rho).$
\end{Proposition}

\begin{proof}Since $\iota$ is an algebra morphism and $\tilde\rho$ is an infinitesimal character over $\cHS(\A)$, for any words $w_1,w_2\in T(T_+(\A))$ we have
\[
	\rho(w_1|w_2) = \tilde\rho(\iota(w_1|w_2)) = \tilde\rho(\iota(w_1)\iota(w_2)) = 0.
\]
Furthermore, since $\iota$ is also a coalgebra morphism, we obtain that for any $w\in T_+(\A)$
\[
	\Phi(w) = \big(\tilde\Phi \circ\iota\big)(w)
	= \tilde\Phi(\iota(w))
	= \exp^*(\tilde \rho)(\iota(w)) = \exp^*(\tilde\rho\circ\iota)(w)
	= \exp^*(\rho)(w).
\]
Note that we used that $\iota$ is a coalgebra morphism in the fourth equality.
\end{proof}

By the above proposition and Theorem \ref{thm:linkNCP}, we conclude that if $w=a_1\cdots a_n$, then
\begin{equation}
\label{eq:rhoiota}
\rho(w) = \sum_{t\in \ST(n)} \tilde{\rho}(t\otimes w)
\end{equation}
is the infinitesimal character associated to the monotone cumulants of $a_1,\dots,a_n\in \A$. On the other hand, note that
\[
	\tilde\Phi
	= \exp^*(\tilde\rho) \ \Leftrightarrow \ \tilde\rho
	= \log^*\big(\tilde\Phi\big)
	= \log^*\big(\varepsilon + \big(\tilde\Phi - \varepsilon\big)\big)
	= \sum_{k=1}^\infty \frac{(-1)^{k+1}}{k} \hat\Phi^{* k},
\]
where $\hat\Phi = \tilde\Phi - \varepsilon$. In order to find an expression for the $k$-fold convolution product $\hat\Phi^{*k}$, we will need to describe the so-called Murua coefficients.

\subsection{Murua coefficients}
\label{ssec:Murua}

Recall that given a forest of planar rooted trees, we have a poset structure canonically associated to it. The following coefficients associated to rooted trees have appeared in the analysis of the continuous Baker--Campbell--Hausdorff problem in a Hall basis (see Murua's article~\cite{Murua}).

\begin{Definition}
Let $t$ be a planar rooted tree seen as a poset with $n = |t|$ vertices. For any integer $0 < k< n+1$, we denote by $\omega_k(t)$ to the number of surjective functions $f\colon t\to \{1,\dots,k\}$ such that for any two vertices $v$, $w$ in $t$ with $v<w$, we have that $f(v)<f(w)$. We then define the quantity $\omega(t)$ by
\begin{equation}\label{eq:muruacoeff}
\omega(t) = \sum_{k=1}^n \frac{(-1)^{k+1}}{k} \omega_k(t).
\end{equation}
If $f$ is a forest with trees $t_1,\dots, t_m$, we define $\omega(f) = \omega(t_1)\cdots \omega(t_m).$
\end{Definition}

In the following tables, we write down the values of the $\omega$-map for several rooted trees.
\begin{center}
\begin{tabular}{ |c|c|c|c|c|c|c|c|c|c| }
 \hline
 $t$ & \begin{tikzpicture}[scale=0.5,
level 1/.style={level distance=7mm,sibling distance=7mm}]
\node(0)[solid node,label=above:{}]{};
\end{tikzpicture}
 & \begin{tikzpicture}[scale=0.5,
level 1/.style={level distance=7mm,sibling distance=7mm}]
\node(0)[solid node,label=above:{}]{}
child{node(1)[solid node]{}};
\end{tikzpicture} & \begin{tikzpicture}[scale=0.5,
level 1/.style={level distance=7mm,sibling distance=7mm}]
\node(0)[solid node,label=above:{}]{}
child{node(1)[solid node]{}
child{[black] node(11)[solid node]{}}
};
\end{tikzpicture} & \begin{tikzpicture}[scale=0.5,
level 1/.style={level distance=7mm,sibling distance=7mm}]
\node(0)[solid node,label=above:{}]{}
child{node(1)[solid node]{}
}
child{node(2)[solid node]{}
};
\end{tikzpicture}& \begin{tikzpicture}[scale=0.5,
level 1/.style={level distance=7mm,sibling distance=7mm}]
\node(0)[solid node,label=above:{}]{}
child{node(1)[solid node]{}
child{[black] node(11)[solid node]{}
child{[black] node(111)[solid node]{}}
}
};
\end{tikzpicture}& \begin{tikzpicture}[scale=0.5,
level 1/.style={level distance=7mm,sibling distance=7mm}]
\node(0)[solid node,label=above:{}]{}
child{node(1)[solid node]{}
child{[black] node(11)[solid node]{}}
child{[black] node(111)[solid node]{}}
};
\end{tikzpicture}& \begin{tikzpicture}[scale=0.5,
level 1/.style={level distance=7mm,sibling distance=7mm}]
\node(0)[solid node,label=above:{}]{}
child{node(1)[solid node]{}
child{[black] node(11)[solid node]{}}
}
child{node(2)[solid node]{}
};
\end{tikzpicture}
& \begin{tikzpicture}[scale=0.5,
level 1/.style={level distance=7mm,sibling distance=7mm}]
\node(0)[solid node,label=above:{}]{}
child{node(1)[solid node]{}
}
child{node(2)[solid node]{}
}
child{[black] node(11)[solid node]{}
};
\end{tikzpicture}& \begin{tikzpicture}[scale=0.5,
level 1/.style={level distance=7mm,sibling distance=7mm}]
\node(0)[solid node,label=above:{}]{}
child{node(1)[solid node]{}
}
child{node(2)[solid node]{}
}
child{[black] node(11)[solid node]{}
}
child{[black] node(111)[solid node]{}
};
\end{tikzpicture}
\\ \hline
 $\quad{\omega(t)^{\phantom{I}}}^{\phantom{I}}$ & 1 & $-\frac{1}{2}$& $\frac{1}{3}$& $\frac{1}{6}$ & $-\frac{1}{4}$ & $ -\frac{1}{6}$ & $-\frac{1}{12}$ & $0$ & $-\frac{1}{30}$ \\[0.1cm]
 \hline
\end{tabular}
\end{center}

\begin{center}
\begin{tabular}{|c|c|c|c|c|c|c|c|c|}
\hline
{ $t$ }& \begin{tikzpicture}[scale=0.5,
level 1/.style={level distance=7mm,sibling distance=7mm}]
\node(0)[solid node,label=above:{}]{}
child{node(1)[solid node]{}
child{[black] node(11)[solid node]{}
child{[black] node(111)[solid node]{}
child{[black] node(1111)[solid node]{}}
}
}
};
\end{tikzpicture}& \begin{tikzpicture}[scale=0.5,
level 1/.style={level distance=7mm,sibling distance=7mm}]
\node(0)[solid node,label=above:{}]{}
child{node(1)[solid node]{}
child{[black] node(11)[solid node]{}
child{[black] node(111)[solid node]{}
}
child{[black] node(1111)[solid node]{}}
}
};
\end{tikzpicture}&
 \begin{tikzpicture}[scale=0.5,
level 1/.style={level distance=7mm,sibling distance=7mm}]
\node(0)[solid node,label=above:{}]{}
child{node(1)[solid node]{}
child{[black] node(111)[solid node]{}}
child{[black] node(1111)[solid node]{}}
child{[black] node(11)[solid node]{}}
};
\end{tikzpicture}
 &
 \begin{tikzpicture}[scale=0.5,
level 1/.style={level distance=7mm,sibling distance=7mm}]
\node(0)[solid node,label=above:{}]{}
child{node(1)[solid node]{}
child{[black] node(111)[solid node]{}child{[black] node(11)[solid node]{}}}
child{[black] node(1111)[solid node]{}}
};
\end{tikzpicture}
 &
 \begin{tikzpicture}[scale=0.5,
level 1/.style={level distance=7mm,sibling distance=7mm}]
\node(0)[solid node,label=above:{}]{}
child{node(1)[solid node]{}
child{[black] node(11)[solid node]{}
child{[black] node(111)[solid node]{}
}
}
}
child{[black] node(1111)[solid node]{}};
\end{tikzpicture} &
\begin{tikzpicture}[scale=0.5,
level 1/.style={level distance=7mm,sibling distance=7mm}]
\node(0)[solid node,label=above:{}]{}
child{node(1)[solid node]{}
child{[black] node(11)[solid node]{}
}
}
child{[black] node(1111)[solid node]{}}
child{[black] node(111)[solid node]{}};
\end{tikzpicture} &
\begin{tikzpicture}[scale=0.5,
level 1/.style={level distance=7mm,sibling distance=7mm}]
\node(0)[solid node,label=above:{}]{}
child{node(1)[solid node]{}
child{[black] node(11)[solid node]{}}
child{[black] node(111)[solid node]{}}
}
child{[black] node(1111)[solid node]{}};
\end{tikzpicture}
&\begin{tikzpicture}[scale=0.5,
level 1/.style={level distance=7mm,sibling distance=7mm}]
\node(0)[solid node,label=above:{}]{}
child{node(1)[solid node]{}
child{[black] node(11)[solid node]{}}
}
child{node(2)[solid node]{}
child{[black] node(21)[solid node]{}}
};
\end{tikzpicture} \\
\hline
 $\quad{\omega(t)^{\phantom{I}}}^{\phantom{I}}$& $\frac{1}{5}$ & $\frac{3}{20}$ & $\frac{1}{30}$ & $\frac{1}{10}$ & $\frac{1}{20}$ & $-\frac{1}{60}$ & $\frac{1}{60}$ & $\frac{1}{30}$ \\[0.1cm]
 \hline
\end{tabular}
\end{center}
\vspace{0.5cm}
These coefficients have also appeared in the context of relations between cumulants in noncommutative probability. More precisely, in the recent work \cite{CEFPP}, the following formula was proven, which allows writing the monotone cumulants in terms of the free and Boolean cumulants, and noncrossing partitions.

\begin{Theorem}[\cite{CEFPP}]
\label{thm:relationsc}
Let $(\A,\varphi)$ be a noncommutative probability space and let $\{h_n\}_{n\geq1}$, $\{b_n\}_{n\geq1}$ and $\{r_n\}_{n\geq1}$ be the families of monotone cumulants, Boolean cumulants, and free cumulants, respectively. Then, for any $n \geq1$ and elements $a_1,\dots,a_n \in \A$ we have
\begin{gather}
\label{eq:FreetoMon}
	h_n(a_1,\dots,a_n) = \sum_{\pi\in \NC_{\mathrm{irr}}(n)} (-1)^{|\pi|-1}\omega(t(\pi)) r_\pi(a_1,\dots,a_n),
\\ \label{eq:BootoMon}
	h_n(a_1,\dots,a_n) = \sum_{\pi\in \NC_{\mathrm{irr}}(n)} \omega(t(\pi)) b_\pi(a_1,\dots,a_n).
\end{gather}
\end{Theorem}

\begin{Remark}The formulas that express the monotone cumulants in terms of the free and Boolean cumulants were previously proved in \cite{AHLV} for the univariate case. The authors of~\cite{AHLV} described the coefficient $(-1)^{|\pi|-1}\omega(t(\pi))$ as the linear coefficient of $P_\pi(k)$, where $P_\pi(k)$ counts the number of non-decreasing $k$-colour labelling of the noncrossing partition~$\pi$.
\end{Remark}

\begin{Remark}The main tool of the proof of Theorem \ref{thm:relationsc} is the \textit{pre-Lie Magnus operator}. In~\cite{EFP3} the authors show that~\eqref{eq:shuffleMCRel} implies that $\rho = \Omega(\kappa)$, where
\begin{equation}
 \label{eq:Magnus}
 \Omega(\kappa) = \sum_{n\geq0} \frac{B_n}{n!} L^{(n)}_{\Omega(\kappa)\rhd}(\kappa),
\end{equation}
with $L^{(n)}_{\Omega(\kappa)\rhd}(\kappa) = \Omega(\kappa) \rhd L^{(n-1)}_{\Omega(\kappa)\rhd}(\kappa)$ for $n\geq1$ and $L^{(0)}_{\Omega(\kappa)\rhd}(\kappa)=\kappa$, $B_n$ stands for the $n$-th Bernoulli number, and $\rhd$ is a product on $\mathfrak{g}$ given by $\gamma_1\rhd\gamma_2 = \gamma_1\succ\gamma_2 - \gamma_2\prec\gamma_1$. On the other hand, one can recall that the Baker--Campbell--Hausdorff formula computes the logarithm of the fundamental solution of a matrix-valued linear differential equation $X'(t) = A(t)X (t)$. In this context, the (classical) Magnus expansion, defined in the same fashion that~\eqref{eq:Magnus} but using the iterated commutators, leads to the Magnus formula that computes the derivative of the logarithm of the fundamental solution of such an equation~\cite{Iserles}. In addition, it is known that Murua coefficients appear as the coefficients of the pre-Lie Magnus expansion of the generator of the free pre-Lie algebra of rooted trees~\cite{CelPat}. Thus, we can think that the appearance of Murua coefficients is, in some sense, expected due to its relations with the Baker--Campbell--Hausdorff formula.\looseness=-1
\end{Remark}

\subsection[The description of $\hat\Phi^{*k}$]{The description of $\boldsymbol{\hat\Phi^{*k}}$}\label{ssec:convolpower}

Now we are ready to describe $\tilde\rho(t\otimes a_1\cdots a_n)$ appearing in~\eqref{eq:rhoiota}, where $t$ is a Schr\"oder tree with $n+1$ leaves and elements $a_1,\dots, a_n\in \A$. First, we will need some notation.

Let $t$ be a Schr\"oder tree with $n+1$ leaves, $w=a_1\cdots a_n\in T_+(\A)$ and $k \in \{1,\dots, n\}$. Our objective is to compute $\hat\Phi^{*k}(t\otimes a_1\cdots a_n)$. By definition of the convolution product, we have that\looseness=-1
\begin{equation}
\label{eq:iterated}
	\hat\Phi^{*k}(t\otimes w) = m_{\mathbb{C}}^{[k]} \circ \hat\Phi^{\otimes k}\circ \delta_\A^{[k]}(t\otimes w),
\end{equation}
where $\delta_\A^{[k]}$ is the $k$-fold iterated coproduct
\[
	\delta_\A^{[k]} = \big(\delta_\A \otimes \Id_{\cHS(\A)}^{\otimes k-2}\big)\circ \delta_{\A}^{[k-1]},
\]
with $\delta_\A^{[2]} = \delta_\A$ and $\delta_\A^{[1]} = \Id_{\cHS(\A)}$. A similar notation applies to $m_{\mathbb{C}}^{[k]}$.

 We will also need the following combinatorial object associated to a Schr\"oder tree.

\begin{Definition}
Let $t$ be a Schr\"oder tree with $m$ internal vertices. We define the \textit{skeleton of $t$}, denoted by $\sk(t)$, to be the planar rooted tree with $m$ vertices each of them in correspondence with a unique internal vertex of $t$, and such that there is an arrow from $v$ to $w$ in $\sk(t)$ if and only if, the corresponding vertex to $w$ in $t$ is a child of the corresponding vertex to $v$ in $t$.
\end{Definition}

\begin{Example}
Let $t$ be the unlabelled Schr\"oder tree pictured in Figure \ref{fig:DecSchTree}. Its skeleton $\sk(t)$ is given by the tree
\begin{tikzpicture}[scale=0.4,
level 1/.style={level distance=7mm,sibling distance=12mm}]
\node(0)[solid node,label=above:{}]{}
child{node(1)[solid node]{}
child{[black] node(11)[solid node]{}}
}
child{node(2)[solid node]{}
child{[black] node(21)[solid node]{}}
};
\end{tikzpicture}.
\end{Example}

\begin{Remark}\label{rem:generalcuts}For any planar rooted tree $s$, Definition~\ref{def:admcut} can be extended to the notion of \textit{general admissible cut on $s$} as a subset $c$ of the set of vertices of $s$ such that any path from the root to any leaf, there is at most one vertex of the path contained in $c$. Observe that the difference with Definition~\ref{def:admcut} is that we are allowed to take leaves of $s$, not only internal vertices. However, if $t\in \ST$, there is a clear bijection between $\operatorname{Adm}(t)$ and the set of general admissible cuts of $\sk(t)$.
\end{Remark}

Let $t$ be a Schr\"oder tree with $\ell$ internal vertices whose root is denoted by $r$. For each $1\leq k\leq \ell$, define $\AdmT_k(t)$ to be the set of sequences $(c_1,\dots,c_{k-1})$, where, for each $j=1,\dots,k-1$, $c_j$ is a~general admissible cut of
\[
	\cR_{j-1}(\sk(t)) := R_{c_{j-1}}(R_{c_{j-2}}(\cdots R_{c_0}(\sk(t))\cdots)),
\]
with $\cR_0(\sk(t))=R_{c_0}(\sk(t)) := \sk(t)$ and $\cR_{k-1}(\sk(t)) =\{r\}$, such that any of
\[
	\ccP_{j}(\sk(t)):=P_{c_j}(R_{c_{j-1}}(R_{c_{j-2}}(\cdots R_{c_0}(\sk(t))\cdots)))
\]
is a nonempty forest of single-vertex trees.

\begin{Remark}Let $t\in \ST$ with root $r$ and $\tilde{c}= (c_1,\dots,c_{k-1})\!\in\! \AdmT_k(t)$. Since \mbox{$\cR_{k-1}(\sk(t)) = \{r\}$}, we have that $\tilde{c}$ is an ordered partition of the vertex set of $\sk(t)$ minus the root of $t$. Furthermore, since any $\ccP_j(\sk(t))$ is a nonempty forest of single-vertex trees, it is easy to see that $c_j$ is a subset of the set of leaves of $\cR_{j-1}(\sk(t))$ and $ \ccP_j(\sk(t)) = c_j,$ for any $1\leq j <k$.
\end{Remark}

The following lemma proves that the cardinality of the set $\AdmT_k(t)$ is given precisely by~$\omega_k(t)$.

\begin{Lemma}\label{lemma:bijection}Let $t$ be a Schr\"oder tree and $r$ be the roof of $t$. Then there is a bijection between $\AdmT_k(t)$ and the set of strictly order-preserving surjective functions $f\colon \sk(t)\to [k]$.
\end{Lemma}
\begin{proof}For any $(c_1,\dots,c_{k-1})\in \AdmT_k(t)$, we define the function $f\colon \sk(t)\to[k]$ by $f(r) = 1$, and $f(v) = k - j+1$ if $v\in c_{j}$.

We will show that the previous map is well defined. Indeed, observe that is $(c_1,\dots,c_k,\{r\})$ is an ordered partition of $\sk(t)$. Moreover, since every $c_j$ is a nonempty set, we have that $f$ is surjective. On the other hand, take $v$, $w$ vertices of $\sk(t)$ such that $v<w$, i.e., there is a~downward directed path from $v$ to $w$ in $\sk(t)$. The elements of the subsequence $(c_1,\dots,c_j)$ are consecutive prunings of $\sk(t)$ and these are taken from the leaves of $\cR_{j-1}(\sk(t))$ upward the root as $j$ increases. Then we have that $v\in c_i$, $w\in c_j$ and $j<i$, and in particular $f(v) = k-i+1 < k-j+1 = f(w)$. Hence $f$ is strictly order-preserving.

 It is clear that two different elements in $\AdmT_k(t)$ will produce two different functions. It remains to show that our correspondence is onto. To this end, for any $f\colon \sk(t)\to[k]$ strictly order preserving surjective function, define the sequence of sets $(c_1,\dots,c_{k-1})$ by
 \[
 c_j = \{ v\in \sk(t)\colon f(v) = k-j+1\},
 \]
 for any $1\leq j<k$. Now, observe that $c_1$ is a general admissible cut of $\sk(t)$ since $c_1$ is a subset of the set of leaves of $\sk(t)$. Actually, if $c_1$ contains a vertex $v$ that is not a leaf, there exists a~descendant of $v$, namely $w$ such that $k = f(v) < f(w)$ and this is a~contradiction. Hence $c_1$ is a~general admissible cut such that $\ccP_1(\sk(t)) = P_{c_1}(\sk(t)) = c_1$, i.e., $\ccP_{1}(\sk(t))$ is a~nonempty forest of single-vertex trees. Observe that the same argument can be applied to the tree $\cR_1(\sk(t))$ in order to show that $c_2$ is a general admissible cut such that $\ccP_{2}(\sk(t)) = c_2$ is a nonempty forest of single-vertex trees, and in general, $c_j$ is a general admissible cut such that $\ccP_j(\sk(t))= c_j$ is a nonempty forest of single-vertex trees, for any $j=1,\dots,k-1$. Since $f$ is order-preserving, we have that $f^{-1}(1) = \{r\}$. In addition since~$f$ is surjective, we have that $(c_1,\dots,c_{k-1},\{r\})$ is an ordered partition of the set of vertices of $\sk(t)$ with $\cR_{k-1}(\sk(t)) = \{r\}$. The above reasoning shows that $(c_1,\dots,c_{k-1})\in \AdmT_k(t)$ is a tuple such that their corresponding strictly order preserving surjective function is~$f$. Hence the map is onto, and the proof is complete.
\end{proof}

The main lemma of this section, which establishes the connection between the $\omega$ coefficients and the moments of random variables, is the following.

\begin{Lemma}
\label{lemma:main}
Let $(\A,\varphi)$ be a noncommutative probability space. Also, let $t \in \ST(n)$ and a~word $w=a_1\cdots a_n\in T_+(\A)$, for $n\geq1$. If $\tilde{\Phi}$ is the extension of $\varphi$ to a character as defined in~\eqref{eq:tildephi} and $\tilde\rho = \log^*(\tilde{\Phi})$, then we have
\begin{equation} \label{eq:tilderho10}
 \tilde\rho(t\otimes w) = \omega(\sk(t)) \varphi_{\pi(t)}(a_1,\dots,a_n).
\end{equation}
\end{Lemma}

\begin{proof}
Recall that
\[
	\tilde\rho(t\otimes w) = \sum_{k\geq 1} \frac{(-1)^{k+1}}{k} \hat\Phi^{*k} (t\otimes a_1\cdots a_n),
\]
where $\hat\Phi^{*k}$ is defined in \eqref{eq:iterated}. By \eqref{eq:muruacoeff}, it is enough to show that
\begin{equation}\label{eq:enough}
	\hat\Phi^{*k} (t\otimes a_1\cdots a_n) = \omega_k(\sk(t)) \varphi_{\pi(t)}(a_1,\dots,a_n),\qquad\mbox{for any }k\geq1.
\end{equation}
Let $k\geq 1$. First, we observe that according to the definition of $\delta_\A$, the $k$-fold iterated coproduct is of the form
\begin{equation}\label{eq:iteratedcop}
	\delta_\A^{[k]}(t\otimes a_1\cdots a_n) = \sum_{\mbox{\tiny{iterated admissible cuts of $t$}}}
	f_1\otimes\cdots \otimes f_k,
\end{equation}
where $f_1$ is a decorated subtree of $t$, each $f_2,\dots,f_k$ is a forest of decorated subtrees of $t$, and the sum is over iterated admissible cuts, i.e., doing admissible cuts on each iteration of the coproduct. Now observe that for any decorated tree, its evaluation on $\hat\Phi = \tilde\Phi - \varepsilon$ is equal to zero if the tree is empty or is not a~corolla. In particular, a~term $f_1\otimes\cdots \otimes f_k$ in the sum~\eqref{eq:iteratedcop} will produce a~zero contribution on $\tilde\Phi^{*k}$ if any of $f_1,\dots,f_k$ is the empty forest or is a~forest which contains a~tree that is not a corolla. Hence, the only terms that can produce a~non-zero contribution are such that $f_1$ is a corolla, and each of $f_2,\dots,f_k$ is a~forest of corollas. On the other hand, since corollas are associated to single-vertex trees via the skeleton map, by Remark~\ref{rem:generalcuts} the sum in~\eqref{eq:iteratedcop} is done precisely over the set $\AdmT_k(t)$.

Finally, by the definition of $\tilde\Phi$ in~\eqref{eq:tildephi}, the fact that $\hat\Phi$ is multiplicative, Definition \ref{def:pit}, and Lemma \ref{lemma:bijection}, the $k$-fold convolution product is equal to
\begin{align*}
	\hat\Phi^{*k}(t\otimes a_1\cdots a_n)
	&= \sum_{\mbox{\tiny{iterated admissible cuts of $t$}}} \hat\Phi(f_1)\cdots \hat\Phi(f_k)\\
	&= \sum_{\tilde{c}\in \AdmT_k(t)} \varphi_{\pi(t)}(a_1,\dots,a_n)= \omega_k(\sk(t)) \varphi_{\pi(t)}(a_1,\dots,a_n),
\end{align*}
where we get \eqref{eq:enough}.
\end{proof}

Now we are ready to prove Theorem~\ref{thm:mainthm1}.

\begin{proof}[Proof of Theorem \ref{thm:mainthm1}]
Let $\rho\colon T(T_+(\A))\to\mathbb{C}$ be the infinitesimal character associated to the monotone cumulants given by Theorem~\ref{thm:linkNCP}. By~\eqref{eq:rhoiota} and Lemma~\ref{lemma:main} we obtain
\begin{align*}
h_n(a_1,\dots,a_n) &= \rho(a_1\cdots a_n)
= \sum_{t\in \ST(n)} \tilde\rho(t\otimes a_1\cdots a_n)
\\ &= \sum_{t\in \ST(n)} \omega(\sk(t))\varphi_{\pi(t)} (a_1,\dots,a_n).\tag*{\qed}
\end{align*}\renewcommand{\qed}{}
\end{proof}

As mentioned in the introduction, the following is a direct corollary of our main theorem.

\begin{Corollary}
\label{cor:maincor}
Let $(\A,\varphi)$ be a noncommutative probability space and let $\{h_n\colon \A^n\to\mathbb{C}\}_{n\geq1}$ be the family of monotone cumulants of $(\A,\varphi)$. Then, for any $n\geq1$ and $a_1,\dots,a_n\in \A$, we have that
\[
	h_n(a_1,\dots,a_n) = \sum_{\pi\in\NC(n)} \alpha(\pi) \varphi_\pi(a_1,\dots,a_n),
\]
where
\[
	\alpha(\pi)
	= \sum_{\substack{t\in \ST(n)\\ \pi(t)
	= \pi}} \omega(\sk(t)),\qquad\forall\,\pi\in\NC(n).
\]
\end{Corollary}

\tikzset{
solid node/.style={circle,draw,inner sep=1.5,fill=black, minimum size = 0.02em},
hollow node/.style={circle,draw,inner sep=1.5,fill=white}
}

\begin{Example}
We will verify the formula provided by Corollary \ref{cor:maincor} for the case $n=3$. From the monotone moment-cumulant formula \eqref{eq:monMC}, one can get the first cumulants in terms of moments
\begin{gather*}
\varphi_1(a_1) = h_1(a_1),
\\ \varphi_2(a_1,a_2) = h_2(a_1,a_2) + h_1(a_1)h_1(a_2),
\\ \varphi_3(a_1,a_2,a_3) = h_3(a_1,a_2,a_3) + h_2(a_1,a_2)h_1(a_3) + h_2(a_2,a_3)h_1(a_1)\\
\hphantom{\varphi_3(a_1,a_2,a_3) =}{} + \frac{1}{2} h_2(a_1,a_3)h_1(a_2) + h_1(a_1)h_1(a_2)h_1(a_3).
\end{gather*}
Hence we obtain
\begin{align}
h_3(a_1,a_2,a_3) ={}& \varphi_3(a_1,a_2,a_3) - \varphi_1(a_1)\varphi_2(a_2,a_3) - \varphi_1(a_3)\varphi_2(a_2,a_3)\nonumber \\ &{}
-\frac{1}{2}\varphi_1(a_2)\varphi_2(a_1,a_3) + \frac{3}{2} \varphi_1(a_1)\varphi_1(a_2)\varphi_1(a_3). \label{eq:example}
\end{align}
On the other hand, by using Corollary \ref{cor:maincor}, the Schr\"oder trees listed in Figure \ref{fig:st3} and the table with the values of $\omega$, we have that the respective coefficients $\alpha$ are given by
\begin{align*}
&\alpha\big( \begin{tikzpicture}[baseline={([yshift=0.7ex]current bounding box.center)},thick,font=\tiny]
 \path 	(0,0) 		node (a) {\phantom{1}}
 	(0.5,0) 	node (b) {\phantom{1}}
 	(1,0) 		node (c) {\phantom{1}};
 \draw (a) -- +(0,0.60) -| (c);
 \draw (b) -- +(0,0.60) -| (b);
 \end{tikzpicture} \big)= \omega\bigg(\sk\bigg( \begin{tikzpicture}[scale=0.4,
level 1/.style={level distance=7mm,sibling distance=7mm}]
\node(0)[solid node,label=above:{}]{}
child{node(1)[hollow node]{}
}
child{node(2)[hollow node]{}
}
child{[black] node(11)[hollow node]{}
}
child{[black] node(111)[hollow node]{}
};
\end{tikzpicture} \bigg) \bigg) = \omega(\begin{tikzpicture}[scale=0.4,
level 1/.style={level distance=7mm,sibling distance=7mm}]
\node(0)[solid node,label=above:{}]{};\end{tikzpicture}
) =1,\\
&\alpha\big(\begin{tikzpicture}[baseline={([yshift=0.7ex]current bounding box.center)},thick,font=\tiny]
 \path 	(0,0) 		node (a) {\phantom{1}}
 	(0.5,0) 	node (b) {\phantom{1}}
 	(1,0) 		node (c) {\phantom{1}};
 \draw (a) -- +(0,0.60) -| (a);
 \draw (b) -- +(0,0.60) -| (c);
 \end{tikzpicture} \big)= \omega\bigg(\sk\bigg( \begin{tikzpicture}[baseline={([yshift=-1.4ex]current bounding box.center)},scale=0.4,
level 1/.style={level distance=7mm,sibling distance=7mm}]
\node(0)[solid node,label=above:{}]{}
child{node(1)[hollow node]{}
}
child{node(2)[solid node]{}
	child{[black] node(11)[hollow node]{}}
	child{[black] node(111)[hollow node]{}}
	child{[black] node(1111)[hollow node]{}}
};
\end{tikzpicture} \bigg) \bigg)+\omega\bigg(\sk\bigg( \begin{tikzpicture}[baseline={([yshift=-1.4ex]current bounding box.center)},scale=0.4,
level 1/.style={level distance=7mm,sibling distance=7mm}]
\node(0)[solid node,label=above:{}]{}
child{node(1)[solid node]{}
	child{[black] node(11)[hollow node]{}}
	child{[black] node(111)[hollow node]{}}
}
child{node(2)[hollow node]{}
}
child{node(3)[hollow node]{}};
\end{tikzpicture} \bigg) \bigg) = \omega\bigg(\begin{tikzpicture}[baseline={([yshift=-1.4ex]current bounding box.center)},scale=0.4,
level 1/.style={level distance=7mm,sibling distance=7mm}]
\node(0)[solid node,label=above:{}]{}child{node(1)[solid node]{}};\end{tikzpicture}
\bigg) + \omega\bigg(\begin{tikzpicture}[baseline={([yshift=-1.4ex]current bounding box.center)},scale=0.4,
level 1/.style={level distance=7mm,sibling distance=7mm}]
\node(0)[solid node,label=above:{}]{}child{node(1)[solid node]{}};\end{tikzpicture}
\bigg) =-1,
\\
&\alpha\big( \begin{tikzpicture}[baseline={([yshift=0.7ex]current bounding box.center)},thick,font=\tiny]
 \path 	(0,0) 		node (a) {\phantom{1}}
 	(0.5,0) 	node (b) {\phantom{1}}
 	(1,0) 		node (c) {\phantom{1}};
 \draw (a) -- +(0,0.60) -| (b);
 \draw (c) -- +(0,0.60) -| (c);
 \end{tikzpicture} \big)= \omega\bigg(\sk\bigg( \begin{tikzpicture}[baseline={([yshift=-1.4ex]current bounding box.center)},scale=0.4,
level 1/.style={level distance=7mm,sibling distance=7mm}]
\node(0)[solid node,label=above:{}]{}
child{node(1)[solid node]{}
	child{[black] node(11)[hollow node]{}}
	child{[black] node(111)[hollow node]{}}
	child{[black] node(1111)[hollow node]{}}
}
child{node(2)[hollow node]{}
};
\end{tikzpicture} \bigg) \bigg)+\omega\bigg(\sk\bigg( \begin{tikzpicture}[baseline={([yshift=-1.4ex]current bounding box.center)},scale=0.4,
level 1/.style={level distance=7mm,sibling distance=7mm}]
\node(0)[solid node,label=above:{}]{}
child{node(1)[hollow node]{}
	}
child{node(2)[hollow node]{}
}
child{node(3)[solid node]{}
	child{[black] node(11)[hollow node]{}}
	child{[black] node(111)[hollow node]{}}
};
\end{tikzpicture} \bigg) \bigg) = \omega\bigg(\begin{tikzpicture}[baseline={([yshift=-1.4ex]current bounding box.center)},scale=0.4,
level 1/.style={level distance=7mm,sibling distance=7mm}]
\node(0)[solid node,label=above:{}]{}child{node(1)[solid node]{}};\end{tikzpicture}
\bigg) + \omega\bigg(\begin{tikzpicture}[baseline={([yshift=-1.4ex]current bounding box.center)},scale=0.4,
level 1/.style={level distance=7mm,sibling distance=7mm}]
\node(0)[solid node,label=above:{}]{}child{node(1)[solid node]{}};\end{tikzpicture}
\bigg) =-1,
\\
&\alpha\big({ \begin{tikzpicture}[baseline={([yshift=0.7ex]current bounding box.center)},thick,font=\tiny]
 \path 	(0,0) 		node (a) {\phantom{1}}
 	(0.5,0) 	node (b) {\phantom{1}}
 	(1,0) 		node (c) {\phantom{1}};
 \draw (a) -- +(0,0.60) -| (c);
 \draw (b) -- +(0,0.45) -| (b);
 \end{tikzpicture} }\big)= \omega\bigg(\sk\bigg( \begin{tikzpicture}[baseline={([yshift=-1.4ex]current bounding box.center)},scale=0.4,
level 1/.style={level distance=7mm,sibling distance=7mm}]
\node(0)[solid node,label=above:{}]{}
child{node(1)[hollow node]{}
	}
child{node(2)[solid node]{}
	child{[black] node(11)[hollow node]{}}
	child{[black] node(111)[hollow node]{}}
}
child{node(3)[hollow node]{}
};
\end{tikzpicture} \bigg) \bigg) = \omega\bigg(\begin{tikzpicture}[baseline={([yshift=-1.4ex]current bounding box.center)},scale=0.4,
level 1/.style={level distance=7mm,sibling distance=7mm}]
\node(0)[solid node,label=above:{}]{}child{node(1)[solid node]{}};\end{tikzpicture}
\bigg) =-\frac{1}{2},
\\
&\alpha\big(\begin{tikzpicture}[baseline={([yshift=0.7ex]current bounding box.center)},thick,font=\tiny]
 \path 	(0,0) 		node (a) {\phantom{1}}
 	(0.5,0) 	node (b) {\phantom{1}}
 	(1,0) 		node (c) {\phantom{1}};
 \draw (a) -- +(0,0.60) -| (a);
 \draw (b) -- +(0,0.60) -| (b);
 \draw (c) -- +(0,0.60) -| (c);
 \end{tikzpicture} \big)= \omega\bigg(\sk\bigg( \begin{tikzpicture}[baseline={([yshift=-1.4ex]current bounding box.center)},scale=0.4,
level 1/.style={level distance=7mm,sibling distance=7mm}]
\node(0)[solid node,label=above:{}]{}
child{node(1)[solid node]{}
 	child{[black] node(11)[solid node]{}
		child{[black] node(112)[hollow node]{}}
		child{[black] node(1112)[hollow node]{}}
	}
	child{[black] node(111)[hollow node]{}}
	}
child{node(2)[hollow node]{}
};
\end{tikzpicture} \bigg) \bigg)
+ \omega\bigg(\sk\bigg( \begin{tikzpicture}[baseline={([yshift=-1.4ex]current bounding box.center)},scale=0.4,
level 1/.style={level distance=7mm,sibling distance=7mm}]
\node(0)[solid node,label=above:{}]{}
child{node(1)[solid node]{}
 	child{[black] node(11)[hollow node]{}
	}
	child{[black] node(111)[solid node]{}
		child{[black] node(112)[hollow node]{}}
		child{[black] node(1112)[hollow node]{}}
	}	
	}
child{node(2)[hollow node]{}
};
\end{tikzpicture} \bigg) \bigg)
+\omega\bigg(\sk\bigg( \begin{tikzpicture}[baseline={([yshift=-1.4ex]current bounding box.center)},scale=0.4,
level 1/.style={level distance=7mm,sibling distance=7mm}]
\node(0)[solid node,label=above:{}]{}
child{node(1)[hollow node]{}	
	}
child{node(2)[solid node]{}
	child{[black] node(11)[solid node]{}
		child{[black] node(112)[hollow node]{}}
		child{[black] node(1112)[hollow node]{}}
	}
	child{[black] node(111)[hollow node]{}	}
};
\end{tikzpicture} \bigg) \bigg)
\\
& \hphantom{\alpha\big(\begin{tikzpicture}[baseline={([yshift=0.7ex]current bounding box.center)},thick,font=\tiny]
 \path 	(0,0) 		node (a) {\phantom{1}}
 	(0.5,0) 	node (b) {\phantom{1}}
 	(1,0) 		node (c) {\phantom{1}};
 \draw (a) -- +(0,0.60) -| (a);
 \draw (b) -- +(0,0.60) -| (b);
 \draw (c) -- +(0,0.60) -| (c);
 \end{tikzpicture} \big)=}{}
+ \omega\bigg(\sk\bigg( \begin{tikzpicture}[baseline={([yshift=-1.4ex]current bounding box.center)},scale=0.4,
level 1/.style={level distance=7mm,sibling distance=7mm}]
\node(0)[solid node,label=above:{}]{}
child{node(1)[hollow node]{}
	}
child{node(2)[solid node]{}
	child{[black] node(11)[hollow node]{}
	}
	child{[black] node(111)[solid node]{}
		child{[black] node(112)[hollow node]{}}
		child{[black] node(1112)[hollow node]{}}
	}
};
\end{tikzpicture} \bigg) \bigg)
+ \omega\bigg(\sk\bigg( \begin{tikzpicture}[baseline={([yshift=-1.4ex]current bounding box.center)},scale=0.4,
level 1/.style={level distance=8mm,sibling distance=12mm},
level 2/.style={level distance=8mm,sibling distance=6mm},
]
\node(0)[solid node,label=above:{}]{}
child{node(1)[solid node]{}
	child{[black] node(112)[hollow node]{}}
		child{[black] node(1112)[hollow node]{}}
	}
child{node(2)[solid node]{}
	child{[black] node(11)[hollow node]{}
	}
	child{[black] node(111)[hollow node]{}
	}
};
\end{tikzpicture} \bigg) \bigg)
\\
 &
\hphantom{\alpha\big(\begin{tikzpicture}[baseline={([yshift=0.7ex]current bounding box.center)},thick,font=\tiny]
 \path 	(0,0) 		node (a) {\phantom{1}}
 	(0.5,0) 	node (b) {\phantom{1}}
 	(1,0) 		node (c) {\phantom{1}};
 \draw (a) -- +(0,0.60) -| (a);
 \draw (b) -- +(0,0.60) -| (b);
 \draw (c) -- +(0,0.60) -| (c);
 \end{tikzpicture} \big)}{}
= 4 \omega \bigg( \begin{tikzpicture}[baseline={([yshift=-1.4ex]current bounding box.center)},scale=0.4,
level 1/.style={level distance=7mm,sibling distance=7mm}]
\node(0)[solid node,label=above:{}]{}
child{node(1)[solid node]{}
	child{node(2)[solid node]{{}}}
	};\end{tikzpicture}\bigg)
 + \omega \bigg( \begin{tikzpicture}[baseline={([yshift=-1.4ex]current bounding box.center)},scale=0.4,
level 1/.style={level distance=7mm,sibling distance=7mm}]
\node(0)[solid node,label=above:{}]{}
child{node(1)[solid node]{}}
	child{node(2)[solid node]{}
	};\end{tikzpicture}\bigg) = \frac{3}{2}.
\end{align*}
We can observe that the obtained $\alpha$ coefficients match the coefficients appearing in \eqref{eq:example}.
\end{Example}

\begin{Remark}
A different approach to give a formula that writes monotone cumulants in terms of moments is to combine equations \eqref{eq:FreetoMon}, \eqref{eq:BootoMon} with \eqref{eq:freeCM}, \eqref{eq:BoolCM}. For instance, combining \eqref{eq:BootoMon} and \eqref{eq:BoolCM}, we obtain for any $n\geq1$ and $a_1,\dots,a_n\in \A$
\begin{align*}
h_n(a_1,\dots,a_n) &= \sum_{\pi\in \NC_{\mathrm{irr}}(n)} \omega(t(\pi)) \prod_{V\in \pi} b_{|V|}(a_1,\dots,a_n|V)
\\ &= \sum_{\pi\in \NC_{\mathrm{irr}}(n)} \omega(t(\pi)) \prod_{V\in \pi} \sum_{\sigma \in \Int(V)} (-1)^{|\sigma|-1} \varphi_\sigma(a_1,\dots,a_n|V),
\end{align*}
where $\Int(V)$ stands for the set of interval partitions of the set $V$. Observe that if $\pi\in \NC_{\mathrm{irr}}(n)$ and for each $V\in \pi$ we take any $\sigma_V \in \Int(V)$, when we consider the disjoint union $\sigma :=\bigsqcup_{V\in \pi} \sigma_V$ we will obtain a noncrossing partition such that $\sigma\leq \pi$ such that the restriction $\sigma|_V$ on each $V\in \pi$ is an interval partition. If we denote by $ \pi_1 \sqsubseteq\pi_2$ to the fact $\pi_1$ and $\pi_2$ are noncrossing partitions satisfying the two previous conditions, we have that $\sqsubseteq$ is indeed a partial order on $\NC(n)$ (see~\cite{JVer}). With this notation, we have
\begin{align*}
h_n(a_1,\dots,a_n) &= \sum_{\pi\in \NC_{\mathrm{irr}}(n)} \omega(t(\pi)) \sum_{\sigma\sqsubseteq\pi} (-1)^{|\sigma|- |\pi|} \varphi_\sigma(a_1,\dots,a_n)
\\ &= \sum_{\sigma \in \NC(n)} \varphi_\sigma(a_1,\dots,a_n) \sum_{\substack{\pi\in \NC_{\mathrm{irr}}(n)\\\sigma\sqsubseteq\pi}} (-1)^{|\sigma|- |\pi|} \omega(t(\pi)).
\end{align*}
By Corollary \ref{cor:maincor}, we should have that
\begin{equation}\label{eq:alphairrMurua}
\alpha(\pi ) = \sum_{\substack{\sigma\in \NC_{\mathrm{irr}}(n)\\\pi\sqsubseteq\sigma}} (-1)^{|\pi|- |\sigma|} \omega(t(\sigma)).
\end{equation}
Observe that the formula in \eqref{eq:alphairrMurua} is different from ours since, in general, the $\omega$ map evaluated on the tree of nesting of $\sigma$ is not the same that the evaluation of $\omega$ on a Schr\"oder tree whose associated partition is $\sigma$. It is of interest to clarify the relation between these two different types of trees associated to a noncrossing partition.
\end{Remark}

\section{Boolean cumulants in terms of moments via Schr\"oder trees}
\label{sec:boolean}

Theorem \ref{thm:mainthm1} shows that there is a nice formula that allows expressing monotone cumulants in terms of moments as a sum over Schr\"oder trees, namely
\begin{equation}
\label{eq:MonST}
h_n(a_1,\dots,a_n) = \sum_{t\in \ST(n)} \omega(\sk(t)) \varphi_{\pi(t)}(a_1,\dots,a_n),
\end{equation}
for any $a_1,\dots,a_n$ and $n\geq1$. On the other hand, Theorem 7.2 in \cite{JMNT} states the analogous formula for free cumulants:
\begin{equation}
\label{eq:tildekappa}
r_n(a_1,\dots,a_n) = \sum_{t\in \operatorname{PST}(n)} (-1)^{i(t)-1} \varphi_{\pi(t)}(a_1,\dots,a_n),
\end{equation}
with $i(t)$ being the number of internal vertices of $t$ and
\[\operatorname{PST}(n) := \{t\in \ST(n)\colon \mbox{the leftmost subtree of $t$ is a leaf} \}.\]

Notice that \eqref{eq:tildekappa} is rather different than the formula in~\eqref{eq:freeCM}, since for any noncrossing partition $\sigma\in \NC(n)$ there may be several $t\in \operatorname{PST}(n)$ such that $\pi(t)=\sigma$. This finer expression for the cumulants also allows the authors in~\cite{JMNT} to recover the operator-valued analogue of~\eqref{eq:tildekappa}.

We are interested now in obtaining a similar equation for the Boolean case by using the unshuffle algebra structure on $\cHS(\A)$. Associated to the half-coproducts $\delta_{\prec,\A}$ and $\delta_{\succ,\A}$, we have two non-associative products on $\operatorname{Lin}(\cHS(\A),\mathbb{C})$ given by
\begin{gather*}
f \;\tilde{\prec}\; g = m_\mathbb{C}\circ (f\otimes g) \circ \delta_{\prec,\A},
\qquad f\; \tilde{\succ} \;g = m_\mathbb{C}\circ (f\otimes g) \circ \delta_{\succ,\A},
\end{gather*}
for any $f,g\in \operatorname{Lin}(\cHS(\A),\mathbb{C})$. Now, let $(\A,\varphi)$ be a noncommutative probability space and let $\tilde\Phi\colon \cHS(\A)\to\mathbb{C}$ be the character associated to $\varphi$ defined in~\eqref{eq:tildephi}. Since $(\operatorname{Lin}(\cHS(\A),\mathbb{C}),\tilde{\prec},\tilde{\succ})$ is a shuffle algebra, we can consider the infinitesimal character $\tilde{\beta}\colon \cHS(\A)\to\mathbb{C}$ such that
\begin{equation}\label{eq:fixScBool}
\tilde\Phi = \varepsilon_\A + \tilde\Phi\; \tilde\succ\; \tilde\beta.
\end{equation}
We now define the map $\beta\colon T(T_+(\A))\to\mathbb{C}$ by $\beta = \tilde{\beta}\circ \iota$, where $\iota$ is the unshuffle bialgebra morphism defined in~\eqref{eq:iota}. By a similar argument that of the proof of Proposition~\ref{prop:easy}, we conclude that $\beta$ satisfies the fixed-point equation
\[
\Phi = \epsilon + \Phi \succ \beta,
\]
where $\Phi$ is as in Proposition~\ref{prop:easy}. By Theorem~\ref{thm:linkNCP}, it follows that
\[
\beta(w) = \sum_{t\in \ST(|w|)} \tilde{\beta}(t\otimes w)
\]
is the infinitesimal character on $T(T_+(\A))$ associated to the Boolean cumulants. We can use~\eqref{eq:fixScBool} in order to obtain an expression for $\tilde\beta(t\otimes w)$. Indeed, let $t\in \ST(n)$ and $a_1,\dots,a_n\in \A$. Observe that the corresponding sum in the definition of $\delta_{\succ,\A}$ is indexed by
\[\operatorname{Adm}(t) \backslash \operatorname{Adm}_\prec(t) = \{c\in \operatorname{Adm}(t)\colon \mbox{$c$ contains the leftmost leaf of $t$} \}.\]
Recalling that $\tilde \beta$ is an infinitesimal character, we have that
\begin{align*}
\tilde\Phi(t\otimes a_1\cdots a_n) &= \big(\varepsilon_\A + \tilde\Phi\; \tilde\succ\; \tilde\beta\big)(t\otimes a_1\cdots a_n) \\&= \sum_{c\in \operatorname{Adm}(t) \backslash \operatorname{Adm}_\prec(t)} \tilde{\Phi} (R_c(t\otimes a_1\cdots a_n)) \tilde\beta(P_c(t\otimes a_1\cdots a_n))\\
&= \sum_{\substack{c=\{v\}\\\mbox{\tiny{$v$ is an internal vertex belonging}}\\\mbox{\tiny{ the leftmost branch of $t$ }}}} \hspace{-0.7cm}\tilde{\Phi} (R_c(t\otimes a_1\cdots a_n)) \tilde\beta(P_c(t\otimes a_1\cdots a_n))
\end{align*}
since only $c$ with $P_c(t\otimes a_1\cdots a_n)$ being a single tree may produce a non-zero contribution. On the other hand, the definition of $\tilde\Phi$ implies that $\tilde{\Phi} (R_c(t\otimes a_1\cdots a_n)) \tilde\beta(P_c(t\otimes a_1\cdots a_n))=0$ when~$R_c(t)$ is neither a corolla nor the single-vertex tree. Hence, the only possibility for a non-zero contribution in the above sum is that $c$ is the set containing the root of $t$, or $c$ is the set containing the leftmost child of the root. If $t'\otimes a_1\cdots a_m$ is the leftmost decorated subtree of $t\otimes a_1\cdots a_n$, and $t''\otimes a_{m+1}\cdots a_n$ is the tree obtained from $t\otimes a_1\cdots a_n$ when we delete $t'\otimes a_1\cdots a_m$, then\looseness=-1
\begin{equation}
\label{eq:ForBool}
\tilde\Phi(t\otimes a_1\cdots a_n) = \tilde\beta(t\otimes a_1\cdots a_n) + \tilde\Phi(t'' \otimes a_{m+1}\cdots a_n) \tilde{\beta}(t'\otimes a_{1}\cdots a_m).
\end{equation}
Now we can prove the main result of this section.
\begin{Proposition}
Let $(\A,\varphi)$ be a noncommutative probability space. Also, let $t \in \ST(n)$ and a~word $w=a_1\cdots a_n\in T_+(\A)$, for $n\geq1$. If $\tilde{\Phi}$ is the extension of $\varphi$ to a character as defined in~\eqref{eq:tildephi} and $\tilde\beta$ is as defined in~\eqref{eq:fixScBool}, then we have
\begin{equation} \label{eq:tildebeta}
 \tilde\beta(t \otimes w)
	= \begin{cases} (-1)^{i(t)-1}\varphi_{\pi(t)}(a_1,\dots,a_n) & \text{if $t\in \operatorname{BST}(n)$},\\ 0& \text{otherwise}, \end{cases}
\end{equation}
where $\operatorname{BST}(n)$ is the subset of Schr\"oder trees $t$ with $n+1$ leaves such that $t$ is obtained as follows: starting from the single-vertex tree, we graft a corolla on the leftmost leaf of the subtree until we get~$t$.
\end{Proposition}
\begin{proof}We will use induction on $i(t)$. For the base, let $t$ be a decorated Schr\"oder tree with $i(t)=1$, i.e., $t$ is a decorated corolla. Using~\eqref{eq:ForBool} and the definition of $\tilde\Phi$ we have
\[
(-1)^{i(t)-1}\varphi_{1_n}(a_1,\dots, a_n) = \tilde\Phi(t\otimes a_1\cdots a_n)
=\tilde\beta(t\otimes a_1\cdots a_n) + 0,
\]
as desired. For the inductive step, we assume that the result holds for any decorated Schr\"oder tree with less than $k$ internal vertices. Now, we take $t\in \ST(n)$ with $i(t) =k$. Again, by using~\eqref{eq:ForBool} and the definition of $\tilde\Phi$, we have
\[
0= \tilde\Phi(t\otimes a_1\cdots a_n) = \tilde\beta(t\otimes a_1\cdots a_n) + \tilde\Phi(t'' \otimes a_{m+1}\cdots a_n) \tilde{\beta}(t'\otimes a_{1}\cdots a_m),
\]
where $t'$ and $t''$ are the decorated Schr\"oder subtrees previously described. Thus
\[ \tilde\beta(t\otimes a_1\cdots a_n) = -\tilde\Phi(t'' \otimes a_{m+1}\cdots a_n) \tilde{\beta}(t'\otimes a_{1}\cdots a_m).\]
Observe that the condition $t\in \operatorname{BST}(n)$ is equivalent to the fact that $t'\in \operatorname{BST}(m)$ with $i(t)-1$ internal vertices, and $t''$ is a corolla. Hence, if $t\in \operatorname{BST}(n)$, we can use the inductive hypothesis on $t'$ where we have
\begin{align*}
 \tilde\beta(t\otimes a_1\cdots a_n) &= -\varphi(a_{m+1}\cdots a_n) (-1)^{i(t)-2}\varphi_{\pi(t')}(a_1,\dots,a_m)
 \\ &= (-1)^{i(t)-1}\varphi_{\pi(t)}(a_1,\dots,a_n).
\end{align*}
On the other hand, if $t\not\in \operatorname{BST}(n)$, we can have that $t'\not\in \operatorname{BST}(m)$ or $t''$ is not a corolla. In the first case, we use the inductive hypothesis to conclude that $\tilde\beta(t'\otimes a_1\cdots a_m) = 0$, while in the second case we have that $\tilde\Phi(t''\otimes a_{m+1}\cdots a_n)=0$. In any case, we obtain that $\tilde\beta(t\otimes a_1\cdots a_n)=0$ as we wanted to show. This concludes the inductive step and the proof of the proposition as well.\looseness=-1
\end{proof}

The proposition above allows us to conclude that the family of Boolean cumulants of $(\A,\varphi)$ can be written by
\begin{equation}
\label{eq:BoolST}
b_n(a_1,\dots,a_n) = \sum_{t\in \operatorname{BST}(n)} (-1)^{i(t)-1} \varphi_{\pi(t)}(a_1,\dots,a_n), \qquad n\geq1.
\end{equation} Unlike the free case, the latter formula is exactly formula \eqref{eq:BoolMC} the function $t\mapsto\pi(t)$ described in Definition \ref{def:pit} is a clear bijection between $\operatorname{BST}(n)$ such that $\Int(n)$ and $i(t) = |\pi(t)|$.

\begin{Remark}
The free analogue showed in \cite{JMNT} establishes that
\begin{equation}
\label{eq:tildekappa1}
	\tilde\kappa(t \otimes a_1\cdots a_n)
	= \begin{cases} (-1)^{i(t)-1}\varphi_{\pi(t)}(a_1,\dots,a_n) & \text{if $t\in \operatorname{PST}(n)$,}\\ 0& \text{otherwise}, \end{cases}
\end{equation}
where $\tilde\kappa\colon \cHS(\A)\to\mathbb{C}$ is the infinitesimal character satisfying the fixed-point equation
\[\tilde\Phi = \varepsilon_\A + \tilde\kappa \,\tilde\prec\, \tilde\Phi.\]
We can collect the result in \eqref{eq:tildekappa1} together with~\eqref{eq:tilderho10} and~\eqref{eq:tildebeta} in order to obtain an analogue of Theorem \ref{thm:linkNCP} in the unshuffle bialgebra of Schr\"oder trees.
\end{Remark}

\begin{Proposition}Let $(\A,\varphi)$ be a noncommutative probability space and let $\tilde\Phi\colon \cHS(\A)\to\mathbb{C}$ the extension of $\varphi$ defined in \eqref{eq:tildephi}. Let $\big(\tilde\kappa,\tilde\beta,\tilde\rho\big)$ be the triple of infinitesimal characters on $\cHS(\A)$ satisfying the identities \[\tilde\Phi = \mathcal{E}_{\tilde\prec}(\tilde\kappa)= \mathcal{E}_{\tilde\succ}\big(\tilde\beta\big) = \exp^*(\tilde\rho).\]
Then, for any $t\in \ST(n)$ and $a_1,\dots,a_n\in\mathcal{A}$, the triple is given by
\begin{gather*}
	\tilde\kappa(t \otimes a_1\cdots a_n)
	= \begin{cases} (-1)^{i(t)-1}\varphi_{\pi(t)}(a_1,\dots,a_n) & \text{if $t\in \operatorname{PST}(n)$,}\\ 0& \text{otherwise,} \end{cases}
	\\ \tilde\beta(t \otimes a_1\cdots a_n)
	= \begin{cases} (-1)^{i(t)-1}\varphi_{\pi(t)}(a_1,\dots,a_n) & \text{if $t\in \operatorname{BST}(n)$,}\\ 0& \text{otherwise,} \end{cases}
	\\ \tilde\rho(t\otimes a_1\cdots a_n) = \omega(\sk(t)) \varphi_{\pi(t)}(a_1,\dots,a_n).
\end{gather*}
Moreover, the evaluations $\tilde\kappa\circ \iota$, $\tilde\beta\circ\iota$ and $\tilde\rho\circ \iota$ on a word $w = a_1\cdots a_n$ coincide with the free, Boolean and monotone cumulants of $a_1,\dots,a_n$, respectively:
\begin{gather*}
(\tilde\kappa\circ \iota)(w) = r_n(a_1,\dots,a_n), \qquad (\tilde\beta\circ\iota)(w) = b_n(a_1,\dots,a_n),\\ (\tilde\rho\circ \iota)(w) = h_n(a_1,\dots,a_n),
\end{gather*}
where $\iota$ is the coalgebra morphism defined in~\eqref{eq:iota}.
\end{Proposition}
In particular, the above proposition provides the new moment-cumulant formulas \eqref{eq:MonST}, \eqref{eq:tildekappa} and \eqref{eq:BoolST} as sum of products of moments indexed by Schr\"oder trees instead of noncrossing partitions.

\subsection*{Acknowledgements}
The authors would like to thank Kurusch Ebrahimi-Fard for their valuable comments in the preparation of this manuscript. Octavio Arizmendi received financial support by CONACYT Grant CB-2017-2018-A1-S-9764 ``Matrices Aleatorias y Probabilidad No Conmutativa'' and by the SFB-TRR 195 ``Symbolic Tools in Mathematics and their Application'' of the German Research Foundation (DFG). Adrian Celestino was partially supported by the project Pure Mathematics in Norway, funded by Trond Mohn Foundation and Troms{\o} Research Foundation. We thank the referees for the careful reading of the paper.

\pdfbookmark[1]{References}{ref}
\LastPageEnding

\end{document}